\documentclass[11pt]{amsart}
\usepackage{cite}
\usepackage{graphicx}
\usepackage{latexsym,amsmath,amsfonts,amscd, amsthm}
\usepackage{color}
\usepackage[colorlinks=true]{hyperref}
\hypersetup{urlcolor=blue, citecolor=red}
\usepackage{tikz}
\usetikzlibrary{arrows,backgrounds,snakes}
\usepackage{epsfig}
\usepackage{epstopdf}
\usepackage{dsfont}
\usepackage{amssymb}
\usepackage{xcolor}
\usepackage{comment}
\usepackage{subfigure}
\usepackage{multirow}
\usepackage{mathrsfs}
\usepackage{diagbox}
\usepackage{enumerate}
\usepackage{ragged2e}

\newtheoremstyle{plainNoItalics}{}{}{\normalfont}{}{\bfseries}{.}{ }{}
\theoremstyle{plain}
\newtheorem{thm}{Theorem}[section]

\newtheorem{defn}[thm]{Definition}
\newtheorem{rem}[thm]{Remark}
\newtheorem{prop}[thm]{Proposition}
\newtheorem{exa}[thm]{Example}

\newcommand{\beq}{\begin{equation}}
\newcommand{\eeq}{\end{equation}}
\newcommand{\beqa}{\begin{eqnarray}}
\newcommand{\eeqa}{\end{eqnarray}}
\newcommand{\bit}{\begin{itemize}}
\newcommand{\eit}{\end{itemize}}
\newcommand{\bedef}{\begin{defn}}
\newcommand{\edefn}{\end{defn}}
\newcommand{\bpro}{\begin{prop}}
\newcommand{\epro}{\end{prop}}

\newcommand{\Dt}{\Delta t}

\newcommand{\eps}{\varepsilon}

\newcommand{\bx}{{\bf x}}

\newcommand{\bu}{{\bf u}}
\newcommand{\bm}{{\bf m}}

\newcommand\ds{ \displaystyle }

\newcommand\red[1]{\textcolor{red}{#1}}

\title[SWEs with  a friction]{Asymptotic preserving scheme for the shallow water equations with non-flat bottom topography and Manning friction term}


\begin{document}

\maketitle
\centerline{\author{Guanlan Huang$^3$, 
Sebastiano Boscarino$^1$ 
and Tao Xiong$^2$}}

\footnotetext{Email address:
{{\tt glhuang@fjnu.edu.cn}  (Guanlan Huang),  {\tt boscarino@dmi.unict.it} (Sebastiano Boscarino),  {\tt txiong@xmu.edu.cn} (Tao Xiong)} \\[2mm]
  $^1$ Department of Mathematics and Computer Science, University of Catania, Catania 95125, Italy.\\[1mm]
  $^2$ School of Mathematical Sciences, Xiamen University, Fujian Provincial Key Laboratory of Mathematical Modeling and High-Performance Scientific Computing, Xiamen, Fujian, 361005, P.R. China.\\[1mm]
  $^3$ School of Mathematics and Statistics \& Key Laboratory of Analytical Mathematics and Applications (Ministry of Education) \& Fujian Key Laboratory of Analytical Mathematics and Applications (FJKLAMA) \& Center for Applied Mathematics of Fujian Province (FJNU), Fujian Normal University, 350117 Fuzhou, P.R. China.
}




\begin{abstract}
  In our previous work~\cite{huang2023high}, we proposed a class of high-order asymptotic preserving (AP)  finite difference weighted essentially non-oscillatory (WENO)  schemes for  solving   the shallow water equations (SWEs) with bottom topography and  Manning friction,   utilizing  a penalization technique  inspired  by \cite{boscarino2014high}.
  Although  the added weighted  diffusive term enhanced stability, it increased computational  cost and slowed  down the convergence rate in  the intermediate regime  between convection and diffusion.
  In this paper, we extend our  previous study  by  removing  the penalization while preserving  the AP property.
  To achieve this,  we employ a high order semi-implicit implicit-explicit Runge-Kutta (SI-IMEX-RK)   time  discretization,  coupled  with  the  high-order WENO reconstruction for first-order derivatives  and a central difference scheme  for second-order spatial derivatives.  This combination  yields  a class of  fully high-order schemes.
  Theoretical analysis and numerical experiments demonstrate that the proposed  schemes  retain AP, asymptotically accurate (AA) and well-balanced properties, while  offering  higher computational efficiency compared to  our previous  schemes in~\cite{huang2023high}, especially  in  the intermediate regime between convection and diffusion.
  Moreover, treating the momentum in the friction terms implicitly is essential for preserving the AP property; otherwise, the scheme fails to converge to the limiting equations. This indicates that implicit treatment of Manning friction is necessary for the stability of the method.
\end{abstract}

\vspace{0.1cm}

\section{Introduction}
\label{sec1}
\setcounter{equation}{0}
\setcounter{figure}{0}
\setcounter{table}{0}
In  computational fluid dynamics (CFD) simulations, explicit methods have been  preferred  for their simplicity, particularly in nonlinear hyperbolic systems, which shocks  can be  generated  even if the initial conditions are continue~\cite{leveque2002finite,toro2009riemann},  but they suffer from  stability issues when handling stiff systems, requiring small time steps. While implicit methods, though more stable, are computationally expensive \cite{butcher1964implicit,butcher2016}. 

To balance efficiency and stability, implicit-explicit (IMEX) schemes emerged, allowing  stiff terms to be treated implicitly and  non-stiff terms explicitly \cite{kennedy2019higher}. These schemes are particularly effective for multiscale problems in complex multi-physics systems\cite{jin2022asymptotic}. 
For instance, one major  advantage of IMEX schemes is their  ability to capture the correct behavior during scale transitions,  such as from convection-dominated to diffusion-dominated regimes, 
especially in  the diffusive limit\cite{boscarinoerror,peng2020stability,xiong2022high,zheng2024high}, which can be difficult to address in scheme design. Such behaviour in these schemes is known as  {\em asymptotic preserving} (AP). 
The concept of AP schemes  was  first proposed by Jin et al. in \cite{jiang1996efficient, jin1999efficient} for kinetic equations and has significant  applications across  various physical problems~\cite{degond2005smooth,lemou2008new, degond2011all,haack2012all}.  
For more details, readers can refer to references~\cite{hu2017ap, jin2022asymptotic}.

Among the various applications of IMEX schemes to  shallow water equations (SWEs) with non-flat bottom topography and Manning friction, we propose two approaches that are particularly relevant in simulating these equations numerically. The first approach involves additive IMEX methods, which use a penalization technique to ensure correct convergence of solutions to the limit equation. This approach has been effective in simulating systems like kinetic equations~\cite{peng2020stability,dimarco2013asymptotic,boscarino2013implicit}, radiation transport~\cite{xiong2022high,zheng2024high}, and other similar applications.
In \cite{huang2023high}, the authors applied additive IMEX time discretization with a penalization strategy, combining it with high-order WENO reconstructions proposed by Jiang and Shu~\cite{shu1998essentially,jiang1996efficient} and central difference schemes \cite{boscarino2019high} for spatial discretization.  The penalization technique  inspired by \cite{boscarino2014high} were  introduced by adding a weighted nonlinear diffusive  term to ensure the  AP property of the scheme.
However, the choice of weights affects the stability of the scheme \cite{boscarino2014high,peng2020stability}, and incorrect weights can  lead to  instability. Typically, these weights approach $0$ in convection-dominated regimes and $1$ in diffusion-dominated regimes, but the penalization term can slow the convergence rate in intermediate states between these regimes.

To address this issue, in this paper we adopt the second approach that involves semi-implicit (SI) IMEX framework. The SI-IMEX approach eliminates the need for penalization techniques while ensuring the schemes maintain the AP property.
 This approach has been introduced in \cite{boscarino2014high, boscarino2016high} and it is particularly effective for  hyperbolic systems with  stiff terms  like the shallow water equations~\cite{huang2022high,liu2020well,liu2019asymptotic,bispen2014imex} and the Euler equations~\cite{huang2024FEhigh,bispen2017asymptotic,boscarino2019high,boscarino2022high}. Specifically, SI-IMEX schemes have been successfully applied to systems, including the all-Froude shallow water equations and all-Mach Euler equations~\cite{boscarino2019high, boscarino2022high,huang2024FEhigh,NoelleLuka2012AP,noelle2014weakly,bispen2017asymptotic,huang2023high,bispen2014imex,liu2019asymptotic}. 

In summary, in this paper our main focus is to  design a class of AP SI-IMEX methods for the time discretization to avoid the penalization by adding a weighted diffusive term. We will show that, by removing the penalty term, we reduce computational cost and enhance the overall efficiency of the scheme. Furthermore, we show that the scheme is also {\em asymptotically accurate} (AA), i.e., the scheme maintains its order of temporal accuracy for the limit equation. Finally, for the spatial discretization,  we will continue to use  WENO reconstructions \cite{jiang1996efficient,shu1998essentially} for the first-order derivatives and  the central difference scheme \cite{boscarino2019high} for second-order derivatives.

\subsection{Model equations and diffusive limit}
The shallow water equations (SWEs) with non-flat bottom topography and Manning friction in  dimensionless form are  expressed as~\cite{bulteau2020fully,huang2023high}:
\begin{equation}
  \left\{
  \begin{array}{ll}
    \partial_t h + \nabla \cdot \bm   = 0, \\ [3mm]
    \partial_t\bm + \nabla \cdot \left(\dfrac{\bm\otimes \bm}{h}\right) + \dfrac{1}{\eps^2}\nabla\left(\dfrac{g}{2}h^2\right)= -\dfrac{1}{\eps^2}gh\nabla B - \dfrac{1}{\eps^2}\gamma \bm.
  \end{array}\right.
  \label{SWe_MB4}
\end{equation}
This system  can be derived from the Navier-Stokes equations,  which  are fundamental  in   describing  fluid motion~\cite{stoker1992water,whitham2011linear,de1871theorie,boussinesq1877essai,lamb1932hydrodynamics}, and they also have wide-ranging applications in fields such as  oceanography, engineering, and geophysics\cite{pedlosky1987geophysical,haidvogel2000model,shchepetkin2005regional,griffiths2007internal,hervouet2000high,liang2009adaptive,kanamori1972mechanism,titov1998numerical,xing2005high,kurganov2002central,xing2006high,xing2006new,huang2022high}. 
Here, $h$ represents the   water depth, $\bm$ is the momentum, and $\bu = \frac{\bm}{h}$ denotes the  velocity. 
These  variables are  defined on a time-space domain $(t,\bx)\in\mathbb{R}^+\times\Omega$. 
The gravitational constant is denoted by $g$, and $B(\bx)$ represents the bottom topography, which is independent of time. 
The term $\gamma \bm$ represents the Manning friction,  with  $\gamma = \frac{gk^2|\bm|}{h^{\eta}}$.
Here,  $|\bm|$ denotes the $L_2$ norm of $\bm$,  and $k$ is the Manning coefficient, indicating the intensity of friction exerted by the bottom on the water,  a higher value of $k$ implies greater friction. 
The parameter $\eta$ is set to $7/3$, and $\otimes$ denotes the Kronecker product.

As the {\bf{late-time}} behavior of the system is considered, in  scenarios where friction dominates, the variables $h$ and $\bm$ undergo  long-time simulations as  $\eps$ goes to zero, driving  the system into  a diffusive regime~\cite{bulteau2020fully, huang2023high}.
This  transition is  based on the  Chapman-Enskog expansion \cite{bulteau2020fully,huang2023high}, where the  conserved  variables are   expressed  as follows:
\begin{equation}
	\label{Exp1}
	\left\{
	\begin{array}{ll}
		h(\bx,t) = h_0(\bx,t) +\eps h_1(\bx,t) +\eps^2 h_2(\bx,t) +  \cdots,\\ [3mm]
		\bm(\bx,t) = \bm_0(\bx,t)+ \eps \bm_1(\bx,t) + \eps^2 \bm_2(\bx,t) + \cdots.
	\end{array}\right.
\end{equation}
Introducing  the variable  $H=h+B$ as the surface level,  its Chapman-Enskog  expansion takes the form:
\begin{equation}
  \label{Exp2}
  H(\bx,t) = h_0(\bx,t) + B(\bx) +\eps h_1(\bx,t) +\eps^2 h_2(\bx,t) +  \cdots,
\end{equation}
since   the bottom topology $B(x)$ is  independent of  time.
Substituting   the expansions \eqref{Exp1} and \eqref{Exp2} into the  system \eqref{SWe_MB4} and  taking the limit as  $\eps \rightarrow  0$,
we obtain the leading-order terms of system \eqref{SWe_MB4}:
\begin{subequations}
    \begin{equation}
        \partial_t h_0 + \nabla \cdot \bm_0   = 0,
        \label{lim_E1}
    \end{equation}
    \begin{equation}
        \nabla\left(\frac{g}{2}h_0^2\right) = - gh_0\nabla B - \gamma_0 \bm_0,
        \label{lim_E2}
    \end{equation}
\end{subequations}
with $\gamma_0 = \frac{gk^2|\bm_0|}{h_0^{\eta}}$.
Denoting $H_0(\bx) = h_0(\bx) + B(\bx)$, we can rewrite equation \eqref{lim_E2} as
\begin{equation}
\bm_0 = -\sqrt{\frac{h_0^{\eta+1}}{k^2}}\frac{\nabla H_0}{\sqrt{|\nabla H_0|}}.
\label{lim_E3}
\end{equation}
Substituting \eqref{lim_E3} into \eqref{lim_E1} yields the following limit system,
\begin{subequations}
  \label{lim_E4}
  \begin{align}
&\partial_t h_0 = \nabla\cdot\left( \sqrt{\frac{h_0^{\eta+1}}{k^2}}\frac{\nabla H_0}{\sqrt{|\nabla H_0|}}\right), \label{lim_E4-1}\\
&\nabla(\frac{g}{2}h_0^2) = - gh_0\nabla B - \gamma_0 \bm_0.\label{lim_E4-2}
  \end{align}
\end{subequations}
This  system   can  also be viewed  a class of  the non-stationary $p$-Laplacian,  with  the  general form:
\begin{equation}
\partial_t h_0 - \nabla\cdot\left(|\nabla h_0|^{p-2}\nabla h_0 \right) = 0, \quad p >1,
\end{equation}
where $p=3/2$ (with an additional coefficient of $\sqrt{h_0^{\eta+1}/k^2}$  omitted). 
More details on $p$-Laplacian can be found in~\cite{andreu1999existence,herrero1981asymptotic,kamin1988fundamental,de1999regularity}.
As a degenerate parabolic equation, the $p$-Laplacian imposes parabolic-type time step restrictions on  explicit schemes~\cite{bulteau2020fully}.
Next, we will show that the SI-IMEX time discretization overcomes this issue and accurately captures the correct solutions as $\eps$ transitions  from $\mathcal{O}(1)$ to $0$.

\subsection{Steady state}
For  systems  with source terms, they  tend to    admit  some physically steady states, and  many interesting phenomena can be interpreted  as  small perturbations around  these equilibrium states.
These make   the preservation of steady states   crucial in the design of numerical schemes.
For the equations~\eqref{SWe_MB4}, when system reaches a steady state, it reduces to:
\begin{equation}
  \left\{
   \begin{array}{ll}
        \nabla \cdot \bm   = 0, \\ [3mm]
  \ds   \nabla \cdot \left(\dfrac{\bm\otimes \bm}{h}\right) + \frac{1}{\eps^2}\nabla\left(\dfrac{g}{2}h^2\right)= -\frac{1}{\eps^2}gh\nabla b - \frac{1}{\eps^2}\gamma \bm.
  \end{array}\right.
  \label{SWe_MB6}
\end{equation}
Assuming  that $\mathbf{m} \equiv \mathbf{C}$ is  a constant vector, particularly when  $\mathbf{C}=\mathbf{0}$, this system reaches what is  known as ``lake at rest'' steady state, and  the system~\eqref{SWe_MB6} simplifies  to:
\begin{equation}
  \label{Still-water}
  H=h+B=\text{const}, \quad \bm=\mathbf{0},
\end{equation}
indicating that the system is in  a state of  still water.  Numerical schemes designed for system~\eqref{SWe_MB4} must ensure  this  ``still water'' is exactly  maintained  at the discrete level.  Such schemes are  referred to as  well-balanced schemes.
However, standard numerical schemes often  fail to meet this requirement, necessitating  special modifications to  achieve the well-balanced property.
The pioneering work in this area was  conducted by Bermudez and Vazquez \cite{bermudez1994upwind}. Since then,  various  well-balanced schemes  have been developed across different numerical frameworks, including finite difference, finite volume, and discontinuous Galerkin methods~\cite{leveque1998balancing,zhou2001surface,xing2005high,xing2006high,xing2006new,CGP2006,ern2008well,noelle2006well,rhebergen2008discontinuous}. 
For a comprehensive review, readers may refer to~\cite{Xing14,Kurganov18} for more details.

In more complex scenarios where $\mathbf{C}\neq \mathbf{0}$, representing  a moving equilibrium state, the design of well-balanced schemes becomes  more challenging.
Thus,  we focus solely on  the simpler ``still water'' steady state described in equation~\eqref{Still-water} in this work.

\subsection{Outline}
The remainder  of this paper is organized as follows:
In Section~\ref{sec3}, we design a class of numerical schemes for the SWEs~\eqref{SWe_MB4}.
Section \ref{sec4} provides a detailed analysis of  the AP and AA properties.
Numerical experiments are performed in Section \ref{sec5}, followed by concluding remarks in Section \ref{sec6}.

\section{Numerical scheme}
\label{sec3}
\setcounter{equation}{0}
\setcounter{figure}{0}
In this section, we aim  to design a class of  numerical schemes for the SWEs with bottom topology and Manning friction~\eqref{SWe_MB4} by employing  SI-IMEX time discretizaiton. 
The  parameter $\eps$ is allowed to  range  from $\mathcal{O}(1)$ to $0$. 
When $\eps=\mathcal{O}(1)$, the system is in the  hyperbolic regime, where shocks are  generally generated.
In this regime,  the WENO reconstructions~\cite{jiang1996efficient,shu1998essentially},  used   to discrete  the first-order derivatives,  are  crucial for accurately capturing  these shocks.
As $\eps$ goes to zero,  the  system transitions to a  friction-dominated regime   governed by a nonlinear degenerate parabolic  equation, often referred to as the 'p-Laplacian'~\cite{bulteau2020fully,huang2023high}.
Consequently,  a  central difference scheme is employed  to discrete the second-order derivatives~\cite{boscarino2019high}.
The  high-order SI-IMEX time discretization enables the scheme to overcome  parabolic-type time step restrictions, while the  WENO reconstructions effectively mitigate  the oscillations near discontinuities.
This  combination results in  a class of fully high-order schemes  that can be shown, through both theoretical analysis and numerical experiments,      to possess asymptotic preserving, asymptotically accurate and well-balanced properties. 
The details of these validations will be discussed in the next section.

\subsection{IMEX time discretization}
In numerical discretization strategies, spatial derivatives are typically discretized first, followed by temporal discretization.
However, in this study, we begin with  temporal discretization,  diverging from  the traditional method-of-lines approach.
We first develop  a first-order asymptotic preserving SI-IMEX scheme, then   extend it to  high-order schemes using   high-order IMEX-RK methods.
The detailed   procedures are as follows:
\subsubsection{{\bf First order SI-IMEX time discretization}}
From (\ref{SWe_MB4}) and  variable $H = h + B$, the first-order SI-IMEX scheme is given as follows:
\begin{subequations}
  \label{SI-S1}
  \begin{align} 
    &\frac{h^{n+1} - h^n}{\Dt}  + \nabla\cdot \bm^{n+1} =0,\label{SI-S1-1}\\
    &\frac{\bm^{n+1} - \bm^{n}}{\Dt} + \nabla\cdot\left(\frac{\bm\otimes\bm}{h}\right)^n  = - \frac{\Dt}{\eps^2} gh^{n} \nabla H^{n+1} - \frac{\Dt}{\eps^2}\frac{gk^2}{(h^{n})^{\eta}}\left|\bm^{n+1}\right|\bm^{n+1}, \label{SI-S1-2}
  \end{align}
\end{subequations}
Here, $h^n$ and $\bm^n$ represent the solutions at time $t^n$, treated explicitly in the temporal discretization. Meanwhile, $h^{n+1}$ and $\bm^{n+1}$, the solutions at time  $t^{n+1}$, are treated implicitly and require solving.

To solve the  system \eqref{SI-S1}, we first determine the momentum $\bm^{n+1}$ from the second sub-equation \eqref{SI-S1-2} of the system  as follows:
\begin{equation}
  \label{SI-S2}
  \bm^{n+1} = \bm^n - \Dt \nabla\cdot\left(\frac{\bm\otimes\bm}{h}\right)^n  - \frac{\Dt}{\eps^2} gh^{n} \nabla H^{n+1} - \frac{\Dt}{\eps^2}\frac{gk^2}{(h^{n})^{\eta}}\left|\bm^{n+1}\right|\bm^{n+1},
\end{equation}
Notably,  Eq.~\eqref{SI-S2}  is  nonlinear  in terms of  $\bm^{n+1}$.
To solve it, we first isolate the term involving $\bm^{n+1}$ by  moving the friction term to the left:
\begin{equation}
  \label{SI-S3}
\left(1 + \frac{\Dt }{\eps^2}\,\frac{gk^2\left|\bm^{n+1}\right|}{(h^{n})^{\eta}}\right)\bm^{n+1}
= \bm_{\star} - \frac{\Dt}{\eps^2}gh^{n}\nabla H^{n+1},
\end{equation}
with 
\begin{equation}
  \bm^{\star} = \bm^n - \Dt\nabla\cdot\left(\frac{\bm\otimes\bm}{h} \right)^n.
\end{equation}
Next, taking  the  $L_2$ norm of both sides of equation \eqref{SI-S3} leads to a  quadratic equation for $|\bm^{n+1}|$:
\begin{equation}
  \label{SI-S4}
\left(1 + \frac{\Dt }{\eps^2}\,\frac{gk^2\left|\bm^{n+1}\right|}{(h^{n})^{\eta}}\right)\left|\bm^{n+1}\right|
= \left|\bm_{\star} - \frac{\Dt}{\eps^2}gh^{n}\nabla H^{n+1}\right|.
\end{equation}
It is obviously that Eq.~\eqref{SI-S4} has two real roots, and we  select the positive root to ensure  non-negative of the form:
\begin{equation}
  \label{SI-S5}
  \left|\bm^{n+1}\right| = \frac{-\eps^2+\sqrt{\eps^4+\frac{4\Dt gk^2}{(h^{n})^{\eta}}|\eps^2\bm_{\star} - \Dt gh^{n}\nabla H^{n+1}|}}{\frac{2\Dt gk^2}{(h^{n})^{\eta}}}.
  \end{equation}
  Substituting Eq.~\eqref{SI-S5} into Eq.~\eqref{SI-S3} yields the value of  $\bm^{n+1}$: 
  \begin{equation}
  \label{SI-qn1}
  \bm^{n+1} = \frac{2(\eps^2\bm_{\star} - \Dt g h^{n} \nabla H^{n+1})}{\eps^2+\sqrt{\eps^4+ \frac{4\Dt gk^2}{(h^{n})^{\eta}}|\eps^2\bm_{\star} - \Dt gh^{n}\nabla H^{n+1}|}}.
\end{equation}
We now   substitute equation \eqref{SI-qn1} into equation \eqref{SI-S1-1}, resulting in the following update for $h^{n+1}$:
\begin{equation} 
  \label{SI-hn}
  h^{n+1} = h^n  - \Dt\nabla \cdot \left\{ \frac{2(\eps^2\bm_{\star} - \Dt g h^{n} \nabla H^{n+1})}{\eps^2+\sqrt{\eps^4+ \frac{4\Dt gk^2}{(h^{n})^{\eta}}|\eps^2\bm_{\star} - \Dt gh^{n}\nabla H^{n+1}|}} \right\}.
\end{equation}
The scheme~\eqref{SI-hn} can be proven to converge to the following form:
\begin{equation} 
  h^{n+1}_0 = h^n_0  + \Dt^2\nabla \cdot \left\{ \sqrt{\frac{(h^{n }_0)^{\eta+1}}{k^2}} \frac{\nabla H^{n+1}_0}{\sqrt{|\nabla H^{n+1}_0|}} \right\},
\end{equation}
demonstrating  that scheme~\eqref{SI-hn} is a first-order scheme for the diffusive limit~\eqref{lim_E4}. Therefore, the scheme  maintains   the asymptotic preserving property as $\eps$ goes to $0$, with  detailed validation  provided in Section~\ref{sec4}.
While if we treat the $L_2$ norm of momentum $|\bm|$ explicitly, i.e., $|\bm^n|$ explicitly,  it will break the AP property of the scheme, as discussed in  Remark~\ref{REM-1}.

Now, we want to  solve  Eq.~\eqref{SI-hn} to obtain the value of $h^{n+1}$, and then substitute it into~\eqref{SI-qn1} to update the value of $\bm^{n+1}$. This process completes the update of the numerical solutions at time $t^{n+1}$.
Since  $H^{n+1} = h^{n+1} + B$ depends linearly  on $h^{n+1}$, Eq.~\eqref{SI-hn} can be viewed  as a nonlinear equation in terms of $h^{n+1}$.
We solve this equation using  Picard iteration, a method  successfully   employed  in our previous work \cite{huang2023high}, and  we only give a brief review here.
Additionally, when the system  reaches or is near a steady state,  where  $\nabla H \approx \mathbf{0}$,  this  results in $\eps^2\bm^{\star} - \Dt gh^{n}\nabla H^{n+1} \approx \mathbf{0}$.
To maintain equilibrium state,  consistent  treatment of $\nabla H$ is required in both  numerator and denominator of  Eq.~\eqref{SI-hn}.
If not, as $\varepsilon$ is small enough and  the denominator approaches zero simultaneously,  it may lead  to a   breakdown of the scheme.

Therefore, the iterations proceed  as follows:
\begin{enumerate}
  \item {\bf{Check for Steady State:}} Calculate  $||\nabla H||_1$ to verify whether the system has reached a steady state. If $||\nabla H||_1 < \delta$ (where $\delta$ is a  small tolerance, e.g., $\delta = 10^{-9}$  unless otherwise specified), the scheme is terminated and we set $h^{n+1}=h^n$ to ensure the well-balanced property of the scheme.  Here, $||\nabla H||_1$ is the $L_1$ norm of $\nabla H$. 
  \item {\bf{Continue the Iteration:}} the iteration continues with the following step, 
        \begin{equation}
          h^{n+1,k+1} = h^n - \Dt\nabla \cdot \left\{ \frac{2(\eps^2\bm_{\star} - \Dt g h^{n} \nabla H^{n+1,k+1})}{\eps^2+\sqrt{\eps^4+ \frac{4\Dt gk^2}{(h^{n})^{\eta}}|\eps^2\bm_{\star} - \Dt gh^{n}\nabla H^{n+1,k}|}} \right\}, \quad k \ge 0,
        \end{equation}
        with $h^{n+1,0} = h^n$. The iteration will   stop  until $||h^{n+1,k+1} - h^{n+1,k}||_{1} < \delta$, where $\delta$ is the same constant as mentioned  above.  We then set $h^{n+1}= h^{n+1,k+1}$.
\end{enumerate}

Now that we have updated the value of $h^{n+1}$,   we proceed to  solve for  $\bm^{n+1}$ using  Eq.~\eqref{SI-qn1}. The formatting is now complete.

\begin{rem}
  \label{REM-1}
  The  implicit treatment of the $L_2$ norm of momentum $|\bm^{n+1}|$ in the Manning friction term is necessary. Otherwise, if  we treat it explicitly,  as shown below
  \begin{equation}
    \label{S3-rem-1}
    \bm^{n+1} = \bm^n - \Dt \nabla\cdot\left(\frac{\bm\otimes\bm}{h}\right)^n  - \frac{\Dt}{\eps^2} gh^{n} \nabla H^{n+1} - \frac{\Dt}{\eps^2}\gamma^{n}\bm^{n+1},
  \end{equation}
  with $\gamma^{n} = \frac{gk^2}{(h^{n})^{\eta}}|\bm^{n}|$. Solving this equation yields 
  \begin{equation}
    \label{S3-rem-2}
    \bm^{n+1} = \frac{\eps^2}{\eps^2+\Dt \gamma^{n}}\left\{\bm^{n} - \Dt\nabla\cdot\left(\frac{\bm \otimes \bm}{h}\right)^n - \frac{\Dt}{\eps^2} gh^n\nabla H^{n+1} \right\}.
  \end{equation}
  Now,  substituting Eq.~\eqref{S3-rem-2} into Eq.~\eqref{SI-S1-1}, we obtain:
  \begin{equation}
    \label{S3-rem-3}
    h^{n+1} = h^{\star} + \Dt^2\nabla\cdot\left(\frac{1}{\eps^2+\Dt\gamma^n}gh^n\nabla H^{n+1}\right),
  \end{equation}
  with $h^{\star} = h^{n} - \Dt\nabla \cdot \left\{\frac{\eps^2}{\eps^2+\Dt\gamma^{n}}\left[\bm^{n} - \nabla \cdot\left(\dfrac{\bm\otimes\bm}{h}\right)^n\right]\right\}$.
  As $\eps$ goes to zero, Eq.~\eqref{S3-rem-3} converges to 
  \begin{equation}
    \label{S3-rem-4}
    h^{n+1} = h^{n} +\Dt\nabla\cdot\left(\frac{gh^n\nabla H^{n+1}}{\gamma^n}\right),
  \end{equation}
  it is obvious that Eq.~\eqref{S3-rem-4} cannot converge to the limiting equation \eqref{lim_E4} as $\eps\rightarrow 0$, meaning that the scheme does not have the asymptotic preserving property in the diffusive limit. 
  And the necessity of implicit treatment of the $L_2$ norm of momentum $|\bm^{n+1}|$ in the friction term for maintaining  the asymptotic preserving property is further demonstrated through numerical experiments.
\end{rem}

\subsubsection{{\bf High order  SI-IMEX-RK scheme}}
In this section, we extend the first-order SI-IMEX scheme \eqref{SI-S1} to  high-order SI-IMEX-RK schemes by  employing   the high-order semi-implicit approach presented in \cite{boscarino2016high} . 

We start to consider the general class of autonomous problems of the form
\begin{equation}\label{Eq1}
u' = F(u,u), \quad u(t_0) = u_0.
\end{equation}
where $F: \mathbb{R}^n \times \mathbb{R}^n \to \mathbb{R}^n$ is a sufficiently regular mapping.
\\
We observe that we can rewrite system (\ref{SWe_MB4}) with $H = h+B$, as a partitioned one
\begin{equation}\label{Eq2}
	\left\{
	\begin{array}{l}
\ds		\frac{d u_E}{d t} = F(u_E,u_I),\\[3mm]
\ds                 \frac{d u_I}{d t} = F(u_E,u_I), 
	\end{array}
	\right.
\end{equation}
with initial conditions $u_E(t) = u_I(t) = u_0$ where 
$$
F(u_E,u_I) = \left(\begin{array}{c}
-\nabla m_I\\[2mm]
- \nabla \cdot \left(\dfrac{\bm\otimes \bm}{h}\right)_E - \dfrac{1}{\eps^2}gh_E\nabla H_I - \dfrac{g k^2}{\eps^2 (h_E)^\eta}|\bm_I|\bm_I
\end{array}
\right)
$$
with $u_E =(hu_E,\bm_E)$ and  $u_I =(h_I,\bm_I)$ and $H_I = h_I+B$. 

In such case the solution of system (\ref{Eq2}) satisfies $u(t) = u^{*}(t)$ for any $t \ge t_0$ and is also a solution of Eq.(\ref{Eq1}). System (\ref{Eq2}) is a particular case of partitioned system \cite{hairer1993solving2}, with an additional computational cost since we double the number of variables. Let us remark that the duplication of the unknowns in (\ref{Eq2}) does not take place if appropriate choices of the IMEX-RK scheme are considered  \cite{boscarino2016high}.

Now we are ready to introduce the semi-implicit strategy for solving problems of the form (\ref{Eq2}). This strategy involves treating the first variable $u^*$ explicitly and the second variable $u$ implicitly. In order to do that we shall adopt high-order IMEX R-K schemes and the coefficients of the method are usually represented in a double Butcher $tableau$~\cite{butcher2016} as follows 
\begin{equation}\label{DBT}
  \begin{array}{c|c}
  \tilde{c} & \tilde{A}\\
  \hline
  \vspace{-0.25cm}
  \\
  & \tilde{b}^T \end{array} \ \ \ \ \ \qquad
  \begin{array}{c|c}
  {c} & {A}\\
  \hline
  \vspace{-0.25cm}
  \\
  & {b^T} \end{array},
\end{equation}
where $\tilde{A} = (\tilde{a}_{ij})$, is a $s \times s$ matrix for the explicit scheme, with $\tilde{a}_{ij} = 0$ for $j \ge i$ and  $A = (\tilde{a}_{ij})$ is a $s \times s$ matrix for the implicit one.  The vectors $\tilde{c}=(\tilde{c}_1,...,\tilde{c}_s)^T$, $\tilde{b}=(\tilde{b}_1,..., \tilde{b}_s)^T$, and $c = (c_1, . . . , c_s)^T$ , $b = (b_1, . . . , b_s)^T$ complete the characterization of the scheme. The coefficients $\tilde{c}$ and $c$ are given by the usual relation
\begin{equation}\label{c_c_i}
\tilde{c}_i = \sum_{j = 1}^{i-1}\tilde{a}_{ij}, \quad {c}_i = \sum_{j = 1}^{i}{a}_{ij}.
\end{equation}
For the implicit part, diagonally implicit R-K (DIRK) schemes are often employed \cite{hairer1993solving2}.

Then the SI-IMEX-RK scheme applied to (\ref{Eq2}) is implemented as follows. First we set $u^*_n = u_n$ and compute the stage values  for $i = 1,...,s$,
\begin{align}\label{IntStag}
U^{i}_E = u^n + \Delta t \sum_{j = 1}^{i-1} \tilde{a}_{ij}K_j, \quad U^{i}_I = u^n + \Delta t \sum_{j = 1}^{i} {a}_{ij}K_j,
\end{align}
with the numerical solution
\begin{align}\label{SolSI}
u^{n+1}_E =  u^n + \Delta t {\sum_{i=1}^s} \tilde{b}_i K_i, \quad u^{n+1}_I = u^n + \Delta t {\sum_{i=1}^s} b_i K_i,
\end{align}
where $K_i = F(U^i_E,U^i_I)$, are the RK fluxes. 
\\
Furthermore from now on we shall adopt IMEX R-K schemes with $\tilde{b}_i = b_i$, for $i = 1, ...,s$. We observe that because $\tilde{b}_i = b_i$, for $i = 1, ...,s$ then the numerical solutions are the same, i.e. if $u^*_0 = u_0$ then $u^*_{n} = u_{n}$,  for all $n\ge 0$, therefore the duplication of the system is only apparent \cite{boscarino2016high} and the system reads
\begin{align}\label{IntStag2}
U^{i}_E = u^n + \Delta t \sum_{j = 1}^{i-1} \tilde{a}_{ij}K_j, \quad U^{i}_I = u^n + \Delta t \sum_{j = 1}^{i-1} {a}_{ij}K_j + \Delta t {a}_{ii}K_i,
\end{align}
with the numerical solution
\begin{align}\label{SolSI2}
u^{n+1} = u^n + \Delta t {\sum_{i=1}^s} b_i K_i.
\end{align}
Furthermore, to streamline and efficiently solve the algebraic equations corresponding to the implicit part in the limit case of $\varepsilon \to 0$,  we require that the diagonally implicit Runge-Kutta (DIRK) scheme, is stiffly accurate (SA),  which implies   $b_i = a_{si} = 0\,\, \text{for}\,\, i = 1, \cdots s$. An example of IMEX-RK scheme SA with $b_i = \tilde{b}_i$ is reported later in the numerical tests section. 

In the following we explicitly show how the SI approach applies to equation (\ref{SWe_MB4}), with $H = h+B$.  Assuming that $(h^n,\bm^n)^T$ are known, from (\ref{IntStag2}) and (\ref{SolSI2})  we introduce minor changes to the evaluation of the RK fluxes $K_i$, i.e.,
\begin{align}\label{IntStag2_2}
U^{i}_E = u^n + \Delta t \sum_{j = 1}^{i-1} \tilde{a}_{ij}\bar{K}_j, \quad U^{i}_I = u^n + \Delta t \sum_{j = 1}^{i-1} {a}_{ij}\bar{K}_j + \Delta t {a}_{ii}K_i,
\end{align}
with the numerical solution
\begin{align}\label{SolSI2_2}
u^{n+1} = u^n + \Delta t {\sum_{i=1}^s} b_i \bar{K}_i,
\end{align} 
where we use the new quantity 
$$
\bar{K}_i =  \left(\begin{array}{c}
-\nabla m_I\\[2mm]
- \nabla \cdot \left(\dfrac{\bm\otimes \bm}{h}\right)_E - \dfrac{1}{\eps^2}\nabla\left(\dfrac{1}{2}gh^2\right)_I - \dfrac{1}{\varepsilon^2} g h_I\nabla B - \dfrac{1}{\varepsilon^2}\dfrac{g k^2}{(h_I)^\eta}|\bm_I|\bm_I
\end{array}
\right)
$$
 then the solutions for next time step $t^{n+1}$ can be updated by 
\begin{subequations}
  \label{SI-NS-Hnp1}
  \begin{align}
    &h^{n+1} = h^n - \Dt \sum_{i=1}^{s}b_{i}\nabla \cdot \bm^{(i)}_I,\\[3mm]
    &\bm^{n+1} = \bm^n - \Dt \sum_{i=1}^{s} b_i \left\{\nabla\cdot\left(\frac{\bm\otimes\bm}{h} \right)_E + \frac{1}{\eps^2}\nabla\left(\frac{1}{2}gh^2\right)_I + \frac{1}{\eps^2}gh_I\nabla B + \frac{1}{\eps^2} \frac{gk^2}{h^{\eta}_I}|\bm_I|\bm_I\right\}^{(i)},
  \end{align}
\end{subequations}
with  the values of  each internal  stage: 
\begin{subequations}
\label{SI-NS-HE}
\begin{align}
  &h^{(i)}_E = h^n - \Dt \sum_{j=1}^{i-1}\tilde{a}_{ij}\nabla \cdot \bm^{(j)}_I,\\[3mm]
  &\bm^{(i)}_E = \bm^n - \Dt \sum_{j=1}^{i-1} \tilde{a}_{ij}
  \left\{\nabla\cdot\left(\frac{\bm\otimes\bm}{h} \right)_E + \frac{1}{\eps^2}\nabla\left(\frac{1}{2}gh^2\right)_I + \frac{1}{\eps^2}gh_I\nabla B + \frac{1}{\eps^2} \frac{gk^2}{h^{\eta}_I}|\bm_I|\bm_I\right\}^{(j)}, \label{SI-NS-HE-2}
\end{align}
\end{subequations}
and 
\begin{subequations}
  \label{SI-NS-HI}
  \begin{align}
    &h^{(i)}_I = h^n - \Dt \sum_{j=1}^{i}{a}_{ij}\nabla \cdot \bm^{(j)}_I, \label{SI-NS-HI-1}\\[3mm]
    &\bm^{(i)}_I =\bm^{(i)}_* - \Dt a_{ii} \left\{\nabla\cdot\left(\frac{\bm\otimes\bm}{h} \right)_E + \frac{1}{\eps^2}gh_E\nabla H_I+ \frac{1}{\eps^2} \frac{gk^2}{h^{\eta}_E}|\bm_I|\bm_I\right\}^{(i)}.  \label{SI-NS-HI-2} 
  \end{align}
\end{subequations}
with 
$$
\bm^{(i)}_* = \bm^n - \Dt \sum_{j=1}^{i-1} {a}_{ij}  \left\{\nabla\cdot\left(\frac{\bm\otimes\bm}{h} \right)_E + \frac{1}{\eps^2}\nabla\left(\frac{1}{2}gh^2\right)_I + \frac{1}{\eps^2}gh_I\nabla B + \frac{1}{\eps^2} \frac{gk^2}{h^{\eta}_I}|\bm_I|\bm_I\right\}^{(j)}.
$$
The updates  for the  equations~\eqref{SI-NS-HI} is similar to the  first-order scheme~\eqref{SI-S1} by employing the Picard iteration, detailed process  will be omitted to save space.
\subsection{High order spatial discretizations}\label{Dspace}
For the spatial discretization strategies,  we  employ the    WENO reconstrctions  as proposed by Jiang and Shu \cite{jiang1996efficient,shu1998essentially},  combined with a central difference scheme~\cite{boscarino2019high}.
This choice is motivated by the need to accommodate validations for the parameter $\eps$, which can vary  from $\mathcal{O}(1)$ to $0$.
This flexiblity  enables   the system to  transition from a hyperbolic to a diffusive regime.
Shocks  occurred when $\eps=\mathcal{O}(1)$, while a nonlinear parabolic `p-Laplacian' emerges as $\eps$ goes to zero.
Thus,  WENO reconstrction, utilized for first-order derivatives in   scheme~\eqref{SI-NS-Hnp1} and interal stages~\eqref{SI-NS-HE} and \eqref{SI-NS-HI},  effectively controls  the oscillations near discontinuities, ensuring the stability of the scheme.
Here, we implement  three different types of reconstrctions, that is,  WENO  with  Lax-Friedrichs (LF) splitting $\nabla_{LW }$, WENO reconstrction with non-viscous $\nabla_{W}$ and well-balanced WENO reconstrction $\nabla_{W}^{WB}$.
To ensure the well-balanced property of the scheme,  we make  two minor modifications. First, in the flux splitting for  the conservation of water depth $h$, we replaced the water depth $h$ with the surface level $H$, 
\begin{equation}
 \bm^{\pm} = \frac{1}{2}\left( \bm \pm \Lambda H\right), \quad \text{with} \quad \Lambda = \max_{{\bu,h}} \{|\bu|+\min(1,1/\eps)\sqrt{gh}\},
\end{equation}
which prevents additional errors are introduced by the viscous term at equilibrium state. 
Second, as the system achieves steady state,    the   gradient of pressure will be  balanced by  bottom topology.  Thus,   we  split the bottom topology  into two parts: 
\begin{equation}
  \label{WB-E3}
  gh\nabla B = -gH\nabla B + \nabla \left(\frac{1}{2}gB^2\right), 
\end{equation}
and  we discrete  both  $\nabla B$ and $\nabla \left(\frac{1}{2}gB^2\right)$ terms in the same manner as  $\nabla\left(\frac{1}{2}gh^2\right)$   using  discretizaiton $\nabla_{W}^{WB}$, ensuring that  
\begin{equation}
  \begin{aligned}
    &~ \frac{1}{\eps^2}\nabla_{W}^{WB}  \left(\frac{1}{2}gh^2\right) + \frac{1}{\eps^2}gH\nabla_W^{WB}B - \frac{1}{\eps^2}\nabla_{W}^{WB} \left(\frac{1}{2}gB^2\right) \\[3mm]
    & = \frac{1}{\eps^2}\nabla_{W}^{WB}\left(\frac{1}{2}gh^2  + gHB - \frac{1}{2}gB^2\right)\\[3mm]
     & =\frac{1}{\eps^2}\nabla_{W}^{WB}\left(\frac{1}{2}gH^2\right) \\[3mm]
     & = 0,
  \end{aligned}
\end{equation}
which indicates that  the equilibrium state  is maintained exactly at discrete level.

While for the diffusive limit,  the   high-order centeral  difference scheme is utilized  to approximate second-order derivative terms.
These combinations  are  enssential for  accurately  capturing shocks in the compressible regime and constructing a high-order solver for limiting equations~\eqref{lim_E4},  making   the schemes  suitable for a wide range of  $\eps$ while preserving   the well-balanced property.

\begin{rem}
  Stiffness is encountered  during the update of $\bm_E^{(i)}$ in Eq.~\eqref{SI-NS-HE-2}, primarily due to the amplification of round-off errors, which scale with $\mathcal{O}(\frac{1}{\eps^2})$.
  Although a well-balanced technique is employed  to  balance  the  gradient of pressure, bottom topography, and Manning friction  term in the momentum flux, this amplification ultimately leads to the  breakdown of the scheme.
  In contrast, no such stiffness  occurrs during the  update of  $\bm_I^{(i)}$ in Eq.~\eqref{SI-NS-HI-2}.
  To address this, we focus on updating the implicit solution $\bm_I^{(i)}$. As  long as the diagonal elements $a_{ii}\ (i=1,2,\cdots s)$ of the implicit RK table are non-zero,   the $i$-th layer flux  can be updated  as follows:
  \begin{equation}
  \mathcal{F}_{\bm}^{(i)} = \frac{\bm^n - \bm_I^{(i)} - \sum_{j=1}^{i-1} a_{ij}\Delta t \mathcal{F}_{\bm}^{(j)}}{a_{ii}\Delta t},
  \end{equation}
  where 
  \begin{equation}
  \mathcal{F}_{\bm}^{(j)} =  \left\{\nabla_{LW} \cdot \left(\dfrac{\bm\otimes \bm}{h}\right)_E + \dfrac{1}{\eps^2}\nabla_{W}^{WB}\left(\dfrac{1}{2}gh^2\right)_I + \frac{1}{\eps^2}gH_I\nabla_W^{WB}B - \frac{1}{\eps^2}\nabla_{W}^{WB} \left(\frac{1}{2}gB^2\right) + \dfrac{1}{\varepsilon^2}\dfrac{g k^2}{(h_I)^\eta}|\bm_I|\bm_I\right\}^{(j)},
  \end{equation}
  This update of the flux successfully eliminates the stiffness from the scheme.
\end{rem}

\subsection{\bf{Algorithm flowchart}}
In this part, we will  combine   the high-order IMEX-RK methods with  the finite difference WENO reconstrctions and  the  centeral  difference schemes to develop the following  fully high-order schemes for system \eqref{SWe_MB4}: 
\begin{subequations}
  \begin{align}
    &h^{n+1} = h^n - \Dt \sum_{i=1}^{s}b_{i}\nabla_{LW} \cdot \bm^{(i)}_I,\\[3mm]
    &\bm^{n+1} = \bm^n - \Dt \sum_{i=1}^{s} b_i\mathcal{F}_{\bm}^{(i)},
  \end{align}
\end{subequations}
with   the values of  each interal  stages given by:
\begin{subequations}
\begin{align}
  &h^{(i)}_E = h^n - \Dt \sum_{j=1}^{i-1}\tilde{a}_{ij}\nabla_{LW} \cdot \bm^{(j)}_I,\\[3mm]
  &\bm^{(i)}_E = \bm^n - \Dt \sum_{j=1}^{i-1} \tilde{a}_{ij} \mathcal{F}_{\bm}^{(j)}
\end{align}
\end{subequations}
and 
\begin{subequations}
  \label{SI-NS-HI-final}
  \begin{align}
    &h^{(i)}_I = h^n - \Dt \sum_{j=1}^{i}{a}_{ij}\nabla_{LW} \cdot \bm^{(j)}_I, \label{SI-NS-HI-final-1}\\[3mm]
    &\bm^{(i)}_I = \bm^n - \Dt \sum_{j=1}^{i-1} {a}_{ij}\mathcal{F}_{\bm}^{(j)} - \Dt a_{ii} \left\{\nabla_{LW}\cdot\left(\frac{\bm\otimes\bm}{h} \right)_E + \frac{1}{\eps^2}gh_E\nabla H_I+ \frac{1}{\eps^2} \frac{gk^2}{h^{\eta}_E}|\bm_I|\bm_I\right\}^{(i)}.\label{SI-NS-HI-final-2} 
  \end{align}
\end{subequations}
The update   of equations~\eqref{SI-NS-HI-final} is similar to the  first-order scheme~\eqref{SI-S1}; specifically, we can solve   $\bm_I^{(i)}$ from Eq.~\eqref{SI-NS-HI-final-2}: 
\begin{equation}
  \label{S3-mi-final}
  \bm_I^{(i)} = \frac{2\left(\eps^2\bm_{\star}^{(i)} - a_{ii}\Dt gh_E^{(i)}\nabla H_I^{(i)}\right)}{\eps^2 + \sqrt{\eps^4 + \frac{4a_{ii}\Dt gk^2}{(h_E^{(i)})^{\eta}} |\eps^2 \bm_{\star}^{(i)} - a_{ii}\Dt gh_E^{(i)}\nabla H_I^{(i)}|}},
\end{equation}
where  
\begin{equation*} 
   \bm_{\star}^{(i)} = \bm^n - \Dt \sum_{j=1}^{i-1} {a}_{ij}\mathcal{F}_{\bm}^{(j)} - a_{ii}\Dt \nabla_{LW} \cdot \left( \frac{\bm\otimes\bm}{h}\right)_E^{(i)}.
\end{equation*}
Next, substituting Eq.~\eqref{S3-mi-final} into Eq.~\eqref{SI-NS-HI-final-1} yields 
\begin{equation}
  \label{SI-S3-hi-final}
  h^{(i)}_I = h_{\star}^{(i)} - a_{ii}\Dt\nabla \cdot \left\{ \frac{2\left(\eps^2\bm_{\star}^{(i)} - a_{ii}\Dt gh^{(i)}_E \nabla H^{(i)}_I\right)}{\eps^2 + \sqrt{\eps^4 + \frac{4a_{ii}\Dt gk^2}{(h_E^{(i)})^{\eta}} |\eps^2 \bm_{\star}^{(i)} - a_{ii}\Dt gh_E^{(i)}\nabla H_I^{(i)}|}}\right\}.
 \end{equation} 
 with 
\begin{equation*}
  h_{\star}^{(i)} = h^n - \Dt \sum_{j=1}^{i-1} a_{ij} \nabla_{LW} \cdot \bm_I^{(j)}.
\end{equation*}
We will check whether the steady state is achieved;  if not,  a Picard iteration will be  utilized in  the following form:  
 \begin{equation}
  \label{SI-Picard-S1-final}
  h^{(i),k+1}_I=h_{\star}^{(i)} - a_{ii}\Dt\nabla \cdot \left\{ \frac{2\left(\eps^2\bm_{\star}^{(i)} - a_{ii}\Dt gh^{(i)}_E \nabla H^{(i),k+1}_I\right)}{\eps^2 + \sqrt{\eps^4 + \frac{4a_{ii}\Dt gk^2}{(h_E^{(i)})^{\eta}} |\eps^2 \bm_{\star}^{(i)} - a_{ii}\Dt gh_E^{(i)}\nabla H_I^{(i),k}|}}\right\}, \forall k\geq0,
 \end{equation}
 until $||h^{(i),k+1}_I - h^{(i),k}_I||< \delta$ is satisfied, and we set $h^{(i)}_I=h^{(i),k+1}_I$.
 Here, the second-order derivatives terms can be viewed as $\nabla\cdot\left(a\nabla  H \right)$ with  a nonlinear coefficient $a = \dfrac{gh}{\eps^2 + \sqrt{\eps^4 + \frac{4a_{ii}\Dt gk^2}{h^{\eta}} |\eps^2 \bm_{\star\star}- a_{ii} \Dt gh\nabla H|}}$, this coefficient  is  nonlinear dependent on  $H$, and  we will employ the high-order   central  difference schemes to discretize it.

 After solving $h^{(i)}_I$,  update the $\bm^{(i)}_I$ as follows: 
    \begin{equation}
      \label{mif}
      \bm_I^{(i)} = \frac{2\left\{\eps^2\bm_{\star}^{(i)} - a_{ii}\Dt\nabla_W^{WB}\left(\frac{1}{2}gh_I^2\right)^{(i)} - a_{ii}\Dt gH_I^{(i)}\nabla_{W}^{WB}B + a_{ii}\Dt \nabla_W^{WB}\left(\frac{1}{2}gB^2\right) \right\}}{\eps^2 + \sqrt{\eps^4 + \frac{4a_{ii}\Dt gk^2}{(h_I^{(i)})^{\eta}} |\eps^2 \bm_{\star}^{(i)} - a_{ii}\Dt \nabla_W^{WB}\left(\frac{1}{2}gh_I^2\right)^{(i)} - a_{ii}\Dt gH_I^{(i)}\nabla_{W}^{WB}B + a_{ii}\Dt\nabla_W^{WB}\left(\frac{1}{2}gB^2\right)|}}.
    \end{equation}
Notice that as the  system achieves equilibrium, the value of $\bm_I^{(i)}$ updated by Eq.~\eqref{mif} is expected to $\mathbf{0}$.  However,  if $\eps$ goes to   zero  simultaneously,  instability may arise  due to the occurrence of the  $\frac{0}{0}$ situation.
To prevent this, we implement a cut-off criterion:  if  condition 
\[
\Big|\eps^2 \bm_{\star}^{(i)} - a_{ii}\Dt \nabla_W^{WB}\left(\frac{1}{2}gh_I^2\right)^{(i)} - a_{ii}\Dt g H_I^{(i)}\nabla_{W}^{WB}B + a_{ii}\Dt\nabla_W^{WB}\left(\frac{1}{2}gB^2\right)\Big|<\xi,
\]
and $|\bm^n| < \xi$ is satisfied, where $\xi$ is a round-off constant (set to $\xi=10^{-15}$ unless otherwise specified),  we set   $\bm_I^{(i)}=\mathbf{0}$ to  ensure the system remains in a  steady state.

\section{Asymptotic preserving (AP) and Asymptotically Accurate (AA) property}
\label{sec4}
\setcounter{equation}{0}
\setcounter{figure}{0}
In this section, we aim to validate the asymptotic preserving (AP) property of the firs-order scheme \eqref{SI-S1}, as well as  the asymptotically accurate (AA) property of the high-order schemes \eqref{SI-NS-Hnp1}-\eqref{SI-NS-HI}. 
We focus on the AP and AA analysis on time discretization, while keeping the space continuous.
Additionally, we assume that  all variables can  be expressed  in  the following Chapman-Enskog form:
\begin{equation}
  \label{S4-E1}
  h^{n}(\bx) = h_0^n(\bx) +\mathcal{O}(\eps), \quad \bm^{n}(\bx) = \bm_0^n(\bx) + \mathcal{O}(\eps),
\end{equation}
with $H^n_0(\bx) = h_0^n(\bx) + B(\bx)$, and the leading-order terms,  satisfy the limit equation \eqref{lim_E3}:
\begin{equation}
  \label{S4-E2}
  \bm_0^n = -\left(\sqrt{\frac{h_0^{\eta+1}}{k^2}}\frac{\nabla H_0}{\sqrt{|\nabla H_0|}}\right)^n,
\end{equation}
where $h^n(\bx) = h(\bx,t^n)$ and $\bm^n(\bx) = \bm(\bx,t^n)$.
We assume that the initial data  are well-prepared in the sense of (\ref{Exp1}), i.e., $h^n(\bx) = h(\bx,t^n)$ and $\bm^n(\bx) = \bm(\bx,t^n)$ admit the decomposition  \eqref{S4-E1} and \eqref{S4-E2} at $n=0$.

\subsection{AP property}
We now  present the AP  property of the first-order scheme \eqref{SI-S1} and   provide the complete proof.
\begin{thm}
The first-order scheme \eqref{SI-S1} is AP in the sense that as $\eps\to 0$ the limit scheme is a consistent approximation of the limit equation \eqref{lim_E4}  as long as the initial conditions are well-prepared. 
\end{thm}
\begin{proof}
  This  property can be  proven via  mathematical induction. As $\eps$  approaches zero, the scheme~\eqref{SI-S1} with the  modification \eqref{SI-S2} simplifies to:
  \begin{subequations}
    \label{S4-AP-E1}
    \begin{align} 
      &\frac{h^{n+1}_0 - h^n_0}{\Dt}  + \nabla\cdot \bm^{n+1}_0 =0, \label{S4-AP-E1-1}\\
      &  gh^{n}_0 \nabla H^{n+1}_0 = - \frac{gk^2}{(h^{n}_0)^{\eta}}|\bm^{n+1}_0|\bm^{n+1}_0.\label{S4-AP-E1-2}
    \end{align}
  \end{subequations}
  Solving equation~\eqref{S4-AP-E1-2}, we find:
  \begin{equation}
    \label{S4-AP-E2}
    \bm^{n+1}_0 = -\sqrt{\frac{(h^n_0)^{\eta+1}}{k^2}}\frac{\nabla H^{n+1}_0}{\sqrt{|\nabla H^{n+1}_0|}},
  \end{equation}
  then, substituting Eq.~\eqref{S4-AP-E2} into~Eq.~\eqref{S4-AP-E1-1} yields:
  \begin{equation}
    \label{S4-AP-E3}
    \frac{h^{n+1}_0 - h^n_0}{\Dt}   = \nabla\cdot\left(\sqrt{\frac{(h^n_0)^{\eta+1}}{k^2}}\frac{\nabla H^{n+1}_0}{\sqrt{|\nabla H^{n+1}_0|}}\right),
  \end{equation}
  which is a consistent first-order scheme for the limit equation~\eqref{lim_E4-1}. Given well-prepared   initial conditions, the AP property  holds  for all subsequent  time steps $t^n$.
\end{proof}

\subsection{AA property}
The AP property guarantees only the consistency of the scheme, but in general the AP property does not guarantee the high order accuracy of SI-IMEX schemes as $\varepsilon \to 0$ and the order of accuracy may degrade. Here we now proceed to prove the AA  property of high-order SI-IMEX-RK schemes \eqref{SI-NS-Hnp1}-\eqref{SI-NS-HI} in order to maintain the accuracy as $\varepsilon \to 0$. In what follows, we recognize that the SA condition is crucial to guarantee the AA property of our SI-IMEX-RK scheme.

\begin{thm}
Consider an SI-IMEX-RK scheme  \eqref{SI-NS-Hnp1}-\eqref{SI-NS-HE}-\eqref{SI-NS-HI} of order $p$. Assume that the IMEX-RK method is SA and the initial conditions are well-prepared. Let us denote by  $U^1(\bx;\eps) = \left(h^1(\bx;\eps), \bm^1(\bx;\eps)\right)^T$ the numerical solutions obtained after one time step. Denote $U^{exa}(\bx,t) = (h^{exa}, \bm^{exa})^T$ as the exact solutions of the limit equation~\eqref{lim_E4}. Then after one step, if the errors between the numerical and exact solutions satisfies the  following estimate:
\begin{equation}
  \label{S4-AP-AA}
\lim_{\eps \rightarrow 0} U^1(\bx; \eps) = U^{exa}(\bx, \Delta t) + \mathcal{O}(\Delta t^{p+1}),
\end{equation}
then the schemes possess the  AA  property.
\end{thm}

\begin{proof}
This property can also be proven by mathematical induction for time stages. 
The strategy is as follows. From the AP property of the first order scheme, assuming that conditions \eqref{S4-E1} and \eqref{S4-E2}  hold at time step  $t^n$, we can get scheme \eqref{S4-AP-E2} and \eqref{S4-AP-E3} hold after one time step $a_{11}\Delta t$ at the first stages. Then, if we assume the same convergence holds for all $(i-1)$-th stages, we will prove it also for the $i$-th stage, that is, for  $i=1,2,\cdots,s$.

Therefore, we get: 
\begin{enumerate}
  \item  for the values $(h_E,\bm_E)^T$ as $\eps \to 0$ in~\eqref{SI-NS-HE}, they become: 
        \begin{subequations} 
          \label{S4-AP-E4}
          \begin{align}
            &h^{(i)}_{E,0} = h^n_0 - \Dt \sum_{j=1}^{i-1}\tilde{a}_{ij}\nabla \cdot \bm^{(j)}_{I,0}, \label{S4-AP-E4-1}\\[3mm]
            &\bm^{(i)}_{E,0} = \bm^n_0 - \Dt \sum_{j=1}^{i-1} \tilde{a}_{ij}\nabla\cdot\left(\frac{\bm_0\otimes\bm_0}{h_0} \right)_E,
          \end{align}
        \end{subequations}
        with
        \begin{equation}
          \label{S4-AP-E5}
          \bm^{(j)}_{I,0} = -\left(\sqrt{\frac{\left(h_{E,0}\right)^{\eta+1}}{k^2}}\frac{\nabla H_{I,0}}{\sqrt{|\nabla H_{I,0}|}}\right)^{(j)}.
        \end{equation}
        Then,  substituting Eq.~\eqref{S4-AP-E5} into Eq.~\eqref{S4-AP-E4-1} yields:
        \begin{equation}
          h_{E,0}^{(j)} = h^n_0 + \Dt\sum_{j=1}^{i-1}\tilde{a}_{ij} \nabla \cdot\left(\sqrt{\frac{\left(h_{E,0}\right)^{\eta+1}}{k^2}}\frac{\nabla H_{I,0}}{\sqrt{|\nabla H_{I,0}|}}\right)^{(j)}
        \end{equation}
  \item For the values  $(h_I,\bm_I)^T$ as $\eps \to 0$, first from \eqref{SI-NS-HI-2} we get
  $$
\left( gh_{E,0}\nabla H_{I,0}  + \frac{gk^2}{h_{E,0}^{\eta}}|\bm_{I,0}|\bm_{I,0}\right)^{(i)}=0,
  $$
  so that 
  \begin{equation}\label{m0}
            \bm_{I,0}^{(i)} = - \left(\sqrt{\frac{\left(h_{E,0}\right)^{\eta+1}}{k^2}} \frac{\nabla H_{I,0}}{\sqrt{|\nabla H_{I,0}|}} \right)^{(i)},
          \end{equation}
  and from \eqref{SI-NS-HI-1}
            \label{SI-NS-HI-AA}
            \begin{align} 
              &h^{(i)}_{I,0} = h^n_0 - \Dt \sum_{j=1}^{i}\tilde{a}_{ij}\nabla \cdot \bm^{(j)}_{I,0}, \label{SI-NS-HI-AA-1}.
            \end{align}
          Therefore, substituting \eqref{m0}, Eq.~\eqref{SI-NS-HI-AA-1} becomes 
          \begin{equation}
            h_{I,0}^{(i)} = h_{0}^{n} + \Dt \sum_{j=1}^{i} a_{ij} \nabla \cdot \left(\sqrt{\frac{\left(h_{E,0}\right)^{\eta+1}}{k^2}} \frac{\nabla H_{I,0}}{\sqrt{|\nabla H_{I,0}|}} \right)^{(j)}.
          \end{equation}
           Thus, the convergence also holds for the stage $i-$th.
  \end{enumerate}
  Assuming that the SI-IMEX-RK scheme is SA, then the numerical solution coincides with the last internal stage $s$, and then by setting $i = s$, we can update the solutions $(h^{n+1}, \bm^{n+1})^{T}$  with the final internal stage $(h^{(s)}_I, \bm^{(s)}_I)^{T}$. Furthermore, we get in the limit case $\varepsilon = 0$ a SI-IMEX-RK scheme of order $p$ for the numerical solutions of  the limit equation~\eqref{lim_E4}, that is, the SI-IMEX-RK scheme  \eqref{SI-NS-Hnp1}-\eqref{SI-NS-HE}-\eqref{SI-NS-HI} of order $p$ is AA, and the conclusion \eqref{S4-AP-AA}  is obtained.
\end{proof}

\section{Numerical tests}
\label{sec5}
\setcounter{equation}{0}
\setcounter{figure}{0}
\setcounter{table}{0}
In the previous sections, we developed a class of high-order asymptotic preserving and well-balanced numerical schemes for the shallow water equations with non-flat bottom topography and a Manning friction term \eqref{SWe_MB4}. 
In this section, we will perform  some typical  numerical  experiments  with two main  objectives: first, to assess whether the SI-IMEX temporal discretizaiton simplifies the iterative process compared to additive IMEX time discretizaiton proposed in~\cite{huang2023high};  and second,  to determine the necessity of treating  $|\bm|$ implicitly in the Manning friction term.
If implicit treatment is unnecessary, we can achieve a linear scheme, which is our desired outcome.
To address these objectives, we perform experiments using the following two  schemes.
\begin{enumerate}
   \item  Schemes~\eqref{SI-NS-Hnp1}-\eqref{SI-NS-HE}-\eqref{SI-NS-HI} denoted  as ``SI-S1'';
   \item  Schemes~\eqref{SI-NS-Hnp1}-\eqref{SI-NS-HE}-\eqref{SI-NS-HI} with a explicit treatment for $|\bm|$ in Manning friction, as shown in  Eq.~\eqref{S3-rem-1}, and denoted as ``SI-S2''.
\end{enumerate}
\textcolor{red}{Besides,  we denote the additive IMEX schemes  as ``AD-IMEX'' for simplicity . }

Without loss of generality,  both SI-S1 and SI-S2  schemes employ a third-order stiffly accurate (SA) IMEX-RK  scheme, SI-IMEX(4,4,3), for temporal discretization \cite{boscarino2022high}. The double Butcher tableau for these  schemes are  given by:
\begin{align} \label{IMEX1_(4,4,3)}
&\textrm{\bf Explicit :} \nonumber \\ \vspace{8mm}
&\begin{array}{c|cccc}
0 & 0 & 0 & 0 & 0 \\
\gamma & \gamma & 0 & 0 & 0\\
0.717933260754 & 1.243893189483& -0.525959928729 & 0 & 0\\
1 &   0.630412558153 & 0.786580740199 &  -0.416993298352& 0\\
\hline
0 & 0 & 1.208496649176& -0.644363170684 & \gamma
\end{array},\nonumber \\ \vspace{8mm}
&\textrm{\bf Implicit :}  \\ \vspace{8mm}
&\begin{array}{c|cccc}
\gamma&  \gamma & 0  & 0 & 0\\
\gamma & 0&  \gamma & 0 & 0\\
0.717933260754 &0 & 0.282066739245 & \gamma & 0\\
1 &0 & 1.208496649176& -0.644363170684 & \gamma\\
\hline
&0 & 1.208496649176& -0.644363170684 & \gamma
\end{array},\nonumber
\end{align}
with   $\gamma = 0.435866521508$.
For spatial discretization,  a fifth-order WENO  method   is used to discretize  the first-order derivative terms, while  a fourth-order central  difference scheme is applied  to  the  second-order derivative terms. 
As a result,  fully third-order schemes are obtained.
Additionally,  we consider a wide range of $\eps$  to  verify the AP and AA properties of these schemes.
The time step is defined  as $\Delta t = \text{CFL}\,\Delta x/\Lambda$, where  $\Lambda = \max \{|\bu|+\min(1,1/\eps)\sqrt{gh}\}$, with  $\text{CFL}=0.2$. 

By comparing SI-S1 scheme with the AD-IMEX scheme, we can verify whether the SI-IMEX-RK time discretization offers higher computational efficiency. Additionally, a comparison between the  SI-S1 and SI-S2 schemes will clearly demonstrate whether implicit  treatment of  the $L_2$ norm of momentum $|\bm|$ in the friction term is necessary.
Through the experiments, we can observe that the results obtained by SI-S1  scheme  match well with those  obtained by AD-IMEX scheme in both  hyperbolic and diffusive regimes,  demonstrating  that the SI-S1 scheme exhibits  AP,  AA, and  well-balanced properties.
Moreover, the SI-S1 scheme shows  higher efficiency compared to the AD-IMEX scheme, as indicated by fewer iterations required, especially in the intermediate states.
In contrast, the   SI-S2  scheme fails to simulate  the solutions as $\eps$ goes zero, indicating   that it  lacks  the   AP property.
This suggests that the nonlinearity in the schemes designed for SWEs with non-flat bottom topography and Manning friction~\eqref{SWe_MB4} arises from  the nonlinear Manning friction term, and that  implicit treatment of $|\bm|$ is indeed necessary.
\begin{exa}{
\em
\label{exam57}({\bf{1D accuracy test with linear friction}})
Firstly, we want to test the accuracy of  these  two   different schemes under  varying values of  $\eps$ in a one-dimensional region.
To satisfy the requirements for  AP and AA properties, the initial conditions selected here must adhere to   well-prepared conditions, and it is extremely challenging for us. 
Consequently, we only consider the linear situations, where $\gamma=1$,  reducing  the system~\eqref{SWe_MB4} to:
\begin{equation}
\left\{
 \begin{array}{ll}
  h_t + m_x   = 0, \\ [3mm]
  m_t +  \left(\dfrac{m^2}{h}\right)_x + \dfrac{1}{\eps^2}\left(\dfrac{1}{2}gh^2\right)_x= -\dfrac{1}{\eps^2}ghB_x - \dfrac{1}{\eps^2}m.
\end{array}\right.
\end{equation}
Next, we select  the following initial conditions for testing:  
\begin{equation}
  h(x,0) = \sin(\pi x) + 2,  \qquad m(x,0) = -2\pi\cos(\pi x)\left(\sin(\pi x) + 2 \right).
\end{equation}
Obviously, these initial conditions are well-prepared.
Here, we consider a flat bottom topology, with  $B(x)=0$,  the  gravitational constant  $g=2$,  and the  computational domain  $\Omega=[0,2]$.

To verify  the AA  property  of these schemes,  we choose three different values of $\eps: 1,10^{-2},10^{-6}$ for testing.
Setting $T=0.01$ as  finial time,  we impose  periodic boundary conditions,  using  the numerical solutions  with $N=5120$ uniform  meshes as  reference solutions, as the exact solutions are unknown.
We then calculate  the errors between the numerical solutions and the reference solutions.
The  results for  SI-S1 scheme are presented in Table~\ref{T_exam57_1}.
Since SI-S2 is equivalent to  the SI-S1 scheme in the linear case, the results for SI-S2 are  omitted.
From the results, we can see that for $\eps=1$ and $10^{-6}$, the scheme  achieves  full third-order accuracy. 
Overall,  these results demonstrate that the SI-S1 scheme  exhibits both  AP and AA properties.
\begin{table}[htbp]
  \caption{Example \ref{exam57}. The $L_{1}$ errors and orders  obtained by SI-S1 scheme  with three different values $\eps: 1,10^{-2},10^{-6}$.}
  \begin{center}
		\begin{tabular}{c|c|c|c|c|c|c|c}
			\hline\hline	
		\multicolumn{1}{c|}{\multirow{2}*{}}&\multicolumn{1}{|c|}{\multirow{2}*{N}}&\multicolumn{2}{c|}{ $1$}&\multicolumn{2}{c|}{$10^{-2}$}&\multicolumn{2}{c}{$10^{-6}$}\\
			\cline{3-8}
\multicolumn{1}{c|}{} &\multicolumn{1}{|c|}{} &error& order& error&order&error&order\\  \hline\hline
\multicolumn{1}{c|}{\multirow{6}*{$h$}}
    &40 &     3.38E-04 &       --&     7.35E-06 &       --&     7.41E-06 &     --\\ \cline{2-8}
    &80 &     3.53E-05 &     3.26&     1.13E-06 &     2.70&     1.13E-06 &     2.71\\ \cline{2-8}
   &160 &     1.76E-06 &     4.33&     1.58E-07 &     2.84&     1.56E-07 &     2.87\\ \cline{2-8}
   &320 &     7.26E-08 &     4.60&     1.99E-08 &     2.99&     1.99E-08 &     2.96\\ \cline{2-8}
   &640 &     4.79E-09 &     3.92&     2.52E-09 &     2.98&     2.54E-09 &     2.97\\ \cline{2-8}
   &1280&     4.14E-10 &     3.53&     3.17E-10 &     2.99&     3.16E-10 &     3.01\\ \cline{2-8}
   &2560&     4.42E-11 &     3.23&     3.57E-11 &     3.15&     3.53E-11 &     3.16\\ \hline\hline  
\multicolumn{1}{c|}{\multirow{6}*{$m$}}
   &40 &     3.50E-03 &       --&     2.12E-04 &      -- &     1.94E-04 &     --\\ \cline{2-8}
   &80 &     2.66E-04 &     3.71&     3.47E-05 &     2.61&     2.75E-05 &     2.82\\ \cline{2-8}
  &160 &     1.05E-05 &     4.67&     8.75E-06 &     1.99&     3.02E-06 &     3.18\\ \cline{2-8}
  &320 &     3.68E-07 &     4.83&     1.36E-07 &     2.69&     4.63E-07 &     2.71\\ \cline{2-8}
  &640 &     1.98E-08 &     4.21&     2.34E-07 &     2.53&     6.49E-08 &     2.83\\ \cline{2-8}
  &1280&     1.18E-09 &     4.07&     4.17E-08 &     2.49&     8.11E-09 &     3.00\\ \cline{2-8}
  &2560&     1.00E-10 &     3.56&     5.70E-09 &     2.87&     1.02E-10 &     2.99\\ \hline\hline
\end{tabular}
\end{center}
\label{T_exam57_1}
\end{table}

}
\end{exa}

\begin{exa}
  {\em
  \label{exam12}
  ({\bf 1D accuracy test with nonlinear friction}) 
  Secondly, we seek  to verify the accuracy of these schemes with the nonlinear Manning friction term within the one-dimensional domain.
  To achieve this, we employ the  idea of ``manufactured'' exact solutions, as discussed   in~\cite{yang2021high, huang2023high}.
  The  exact solutions are defined as:
  \begin{equation}
    h(x,t)=2 + \eps^2\sin(\pi (x-t)),\qquad m(x,t) =2 + \eps^2\sin(\pi (x-t)),
    \label{exam12_F1}
  \end{equation}
  which  satisfy the modified shallow water equations 
  \begin{equation}
    \label{ex12_1}
  \left\{
  \begin{array}{ll}
  h_t + m_x = 0,\\ [3mm]
  m_t +\left(\dfrac{m^2}{h}\right)_x + \dfrac{1}{\eps^2}\left(\dfrac{1}{2}gh^2\right)_x = -\dfrac{1}{\eps^2}gk^2\dfrac{|m|m}{h^{7/3}} + \dfrac{1}{\eps^2}gk^2\left[ 2 + \eps^2\sin(\pi(x-t)) \right]^{-1/3} \\[3mm] \hspace{3.8cm} + \dfrac{1}{\eps^2} g (2 +\eps^2\sin( \pi(x-t)))^2_x,
  \end{array}
  \right.
  \end{equation}
  Here,  we also consider a flat bottom topology with $B(x) = 0$,  a Manning coefficient of $k=1$,  a  gravitational constant of $g=1$, and a regular  computational domain $\Omega =[0,2]$.
  
  As noted,  both AP and AA properties are based on the  well-prepared initial conditions,  which  require that  as $\eps$ goes to zero, the  initial conditions satisfy the limiting equation\eqref{lim_E3} and \eqref{lim_E4}.
  \red{However, the additional source term in system \eqref{ex12_1} prevalent the $\eps$  from approaching  zero.} 
  Consequently, the experiment will fail  as $\eps$ approaches zero. To avoid this issue,  we only consider the case where  $\eps=1$.
  For this test, we set finial time to   $T=0.04$  and  compare the errors between numerical solutions and exact solutions.
  The results are showed in  Table~\ref{T_exam12_1}.
  From the table, we  observe that both schemes achieve nearly  fifth-order accuracy, attributed to   the dominance of spatial errors.
  These results confirm that  that both schemes attain the expected high-order accuracy in the compressible regime.
  \begin{table}[htbp]
    \caption{ Example \ref{exam12}. The $L_{1}$ errors and orders  for $h$ and $m$ with $\eps=1$.}
    \begin{center}
      \begin{tabular}{c|c|c|c|c|c}
        \hline\hline	
      \multicolumn{1}{c|}{\multirow{2}*{\text{Scheme}}}&\multicolumn{1}{|c|}{\multirow{2}*{N}}&\multicolumn{2}{c|}{ $h$}&\multicolumn{2}{c}{$m$}\\
        \cline{3-6}
  \multicolumn{1}{c|}{} &\multicolumn{1}{|c|}{} &error& order& error&order\\  \hline\hline
  \multicolumn{1}{c|}{\multirow{8}*{\text{{SI-S1}}}}
      &20 &     8.13E-05 &       --&     1.50E-04 &       --\\ \cline{2-6}
      &40 &     2.44E-06 &     5.06&     3.03E-06 &     5.63\\ \cline{2-6}
      &80 &     7.41E-08 &     5.04&     8.11E-08 &     5.22\\ \cline{2-6}
     &160 &     2.10E-09 &     5.14&     2.40E-09 &     5.08\\ \cline{2-6}
     &320 &     4.02E-11 &     5.71&     7.68E-11 &     4.96\\ \cline{2-6}
     &640 &     2.03E-12 &     4.31&     3.01E-12 &     4.68\\ \cline{2-6}
     &1280&     4.67E-13 &     2.12&     2.26E-13 &     3.74\\ \cline{2-6}
     &2560&     6.51E-14 &     2.84&     2.96E-14 &     2.93\\\hline\hline
    \multicolumn{1}{c|}{\multirow{8}*{\text{{SI-S2}}}}
    &20 &     8.13E-05 &       --&     1.50E-04 &       --\\ \cline{2-6}
    &40 &     2.44E-06 &     5.06&     3.03E-06 &     5.63\\ \cline{2-6}
    &80 &     7.41E-08 &     5.04&     8.11E-08 &     5.22\\ \cline{2-6}
   &160 &     2.10E-09 &     5.14&     2.40E-09 &     5.08\\ \cline{2-6}
   &320 &     4.02E-11 &     5.71&     7.68E-11 &     4.96\\ \cline{2-6}
   &640 &     2.03E-12 &     4.31&     3.01E-12 &     4.68\\ \cline{2-6}
   &1280&     4.67E-13 &     2.12&     2.26E-13 &     3.74\\ \cline{2-6}
   &2560&     6.51E-14 &     2.84&     2.96E-14 &     2.93\\ \hline\hline
  \end{tabular}
  \end{center}
  \label{T_exam12_1}
  \end{table}  

  }\end{exa}

\begin{exa}{
\em
\label{exam55}
In this example, we aim to verify the asymptotic preserving  property of these two  schemes by  considering  both smooth and discontinuous initial conditions, as dissussed in  \cite{bulteau2020fully,huang2023high}.
The smooth initial conditions are defined as:
\begin{equation}
h(x,0) = \left\{
            \begin{aligned}
				&2  ,   & x< -1;\\
                &\frac{1}{2}\left[3+ \sin\left(\frac{3\pi x}{2}\right)\right], & -1\le x < 1;\\
				&1,   & \text{otherwise};
			\end{aligned}
			\right.
\qquad m(x,0) = 0.
\label{exam55_F1}
\end{equation}
The discontinuous initial conditions are given by:
\begin{equation}
h(x,0) = \left\{
            \begin{aligned}
				&2  ,   & x\textless 0;\\
				&1,   & \text{otherwise};
			\end{aligned}
			\right.
\qquad m(x,0) = 0.
\label{exam55_F2}
\end{equation}
Both sets of initial conditions are  tested  within  the domain $[-5,5]$ with  inflow and out outflow boundary conditions.

As described earlier,  stiff behaviors are governed by either  long-time simulations or the domination of friction.
To investigate this,  we  conduct the  following two tests:
\begin{enumerate}
  \item First, we fix the  final time  at $T=0.01$ and set the  coefficient of the friction term to $gk^2=1.0$,  where   $g=9.812$ representing  the gravitational constant, and  select  four different values for  $\eps = 1,\,0.2,\,0.1,\,0.0005$;
  \item Second, we fix the $\eps=1$,  set the final time to $T=0.01\theta$, and adjust the  coefficient of the friction term to  $gk^2=\theta^2$ with $\theta=1,\,5,\,10,\,50$.
\end{enumerate}
We also solve   the limiting equations \eqref{lim_E4} using   the  implicit RK scheme of  SI-IMEX (\ref{IMEX1_(4,4,3)}),  denoted as ``lim''.
To facilitate  effective comparison, we perform  numerical experiments with  two different grid subdivisions, $N=200$ and $N=400$.

The results obtained by the  SI-S1 scheme  are presented in  Fig.~\ref{Fig_ex55_1},  showing  that  as  $\eps$ decreases or the simulation time increases, the numerical solutions converge towards those  obtained from the  limiting equations, indicating that the SI-S1 scheme has the asymptotic preserving property.
Next, we compare the results obtained by these two  different schemes, using the  results from the  AD-IMEX scheme as references.
To conserve space, we focus on two specific cases: $\eps = 0.0005,\,T=0.01,\,gk^2=1$ and $\eps=1,\,T=0.5,\,gk^2=2500$, as shown in Fig.~\ref{Fig_ex55_2} and Fig.~\ref{Fig_ex55_3}.
From these results, we can see that as $\eps$ becomes small, i.e. $\eps=0.0005$, the results obtained with the  SI-S1 scheme  closely match  those from  AD-IMEX scheme. 
However,   the solutions obtained using  SI-S2 scheme exhibit significant oscilations, indicating a failure to    converge to the reference solution. 
This demonstrates that SI-S1 scheme has  the AP property, whereas the  SI-S2 scheme  does not.

Finally, we  compare the efficiency of SI-S1 and AD-IMEX schemes by evaluating  the number of iterations required for  different cases. 
The results, presented in Table~\ref{T_ex55_1}, show that in the hyperbolic regime, the SI-S1 scheme requires a similar number of iterations as the AD-IMEX scheme, 
However, for the intermediate and the diffusive limit, the SI-S1 scheme requires fewer iterations, indicating  higher computational efficiency, particularly in the  discontinuous situation.
All these results demonstrate that the SI-S1 scheme is more  efficient than the AD-IMEX scheme.
\begin{figure}[hbtp]
  \begin{center}
  \mbox{
      \subfigure[water depth $h$]
  {\includegraphics[width=7.0cm]{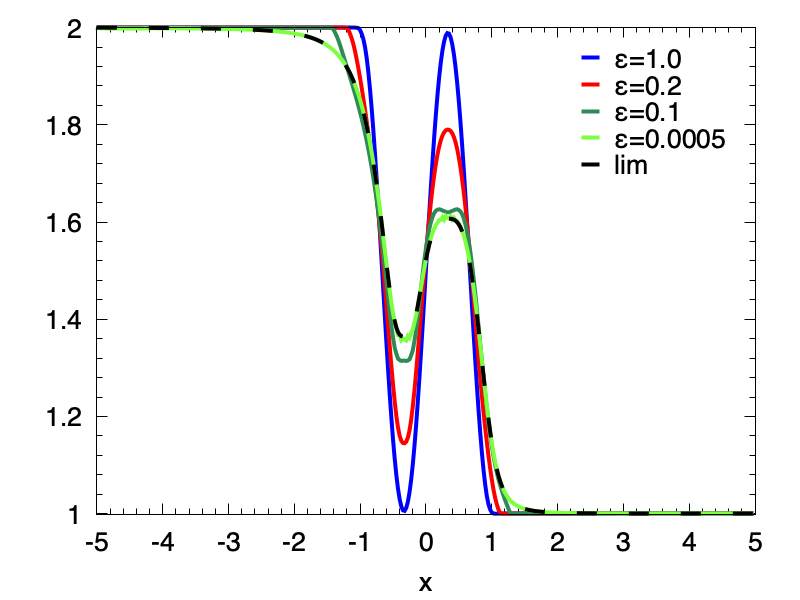}}\quad
      \subfigure[momentum $m$]
  {\includegraphics[width=7.0cm]{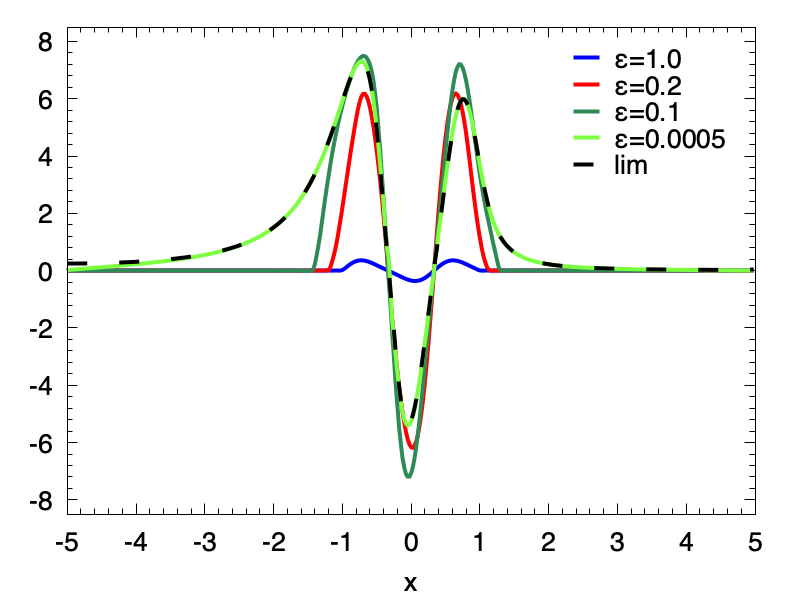}}\quad
  }
  \mbox{
      \subfigure[water depth $h$]
  {\includegraphics[width=7.0cm]{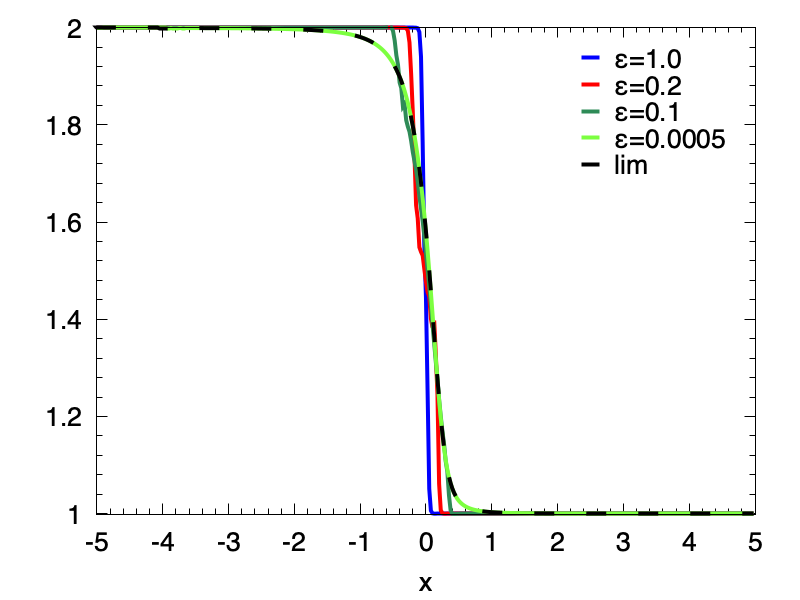}}\quad
      \subfigure[momentum $m$]
  {\includegraphics[width=7.0cm]{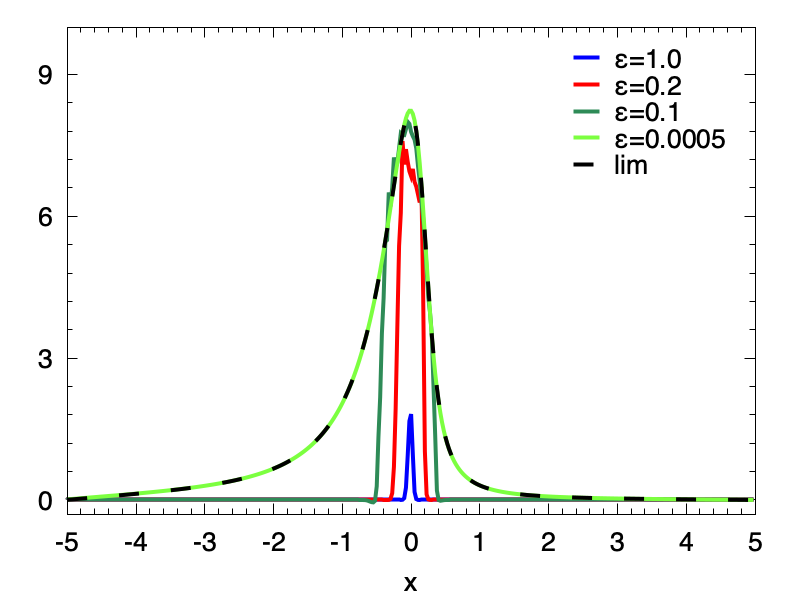}}\quad
  }
  \caption{ Example~\ref{exam55}, the numerical results obtained by SI-S1 for  the water depth  $h$ (left) and the momentum $m$ (right) for the smooth initial conditions \eqref{exam55_F1} and the discontinuous initial conditions \eqref{exam55_F2} with different $\eps$. Top: smooth; Bottom: discontinuous ($N=400$ uniform meshes).}
  \label{Fig_ex55_1}
  \end{center}
  \end{figure}

  \begin{figure}[hbtp]
    \begin{center}
    \mbox{
        \subfigure[ $\eps=0.0005,\, T=0.01,\, gk^2=1$]
    {\includegraphics[width=7.0cm]{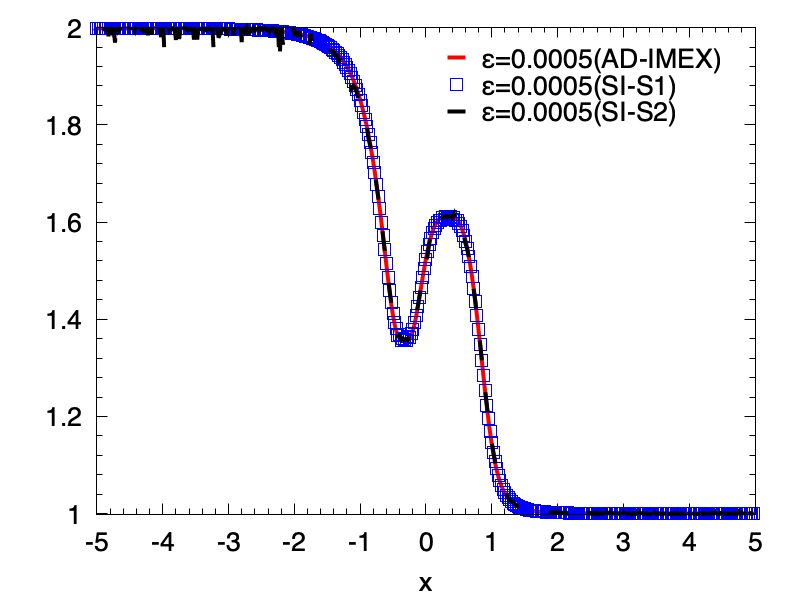}}\quad
        \subfigure[ $\eps=0.0005, T=0.01$]
    {\includegraphics[width=7.0cm]{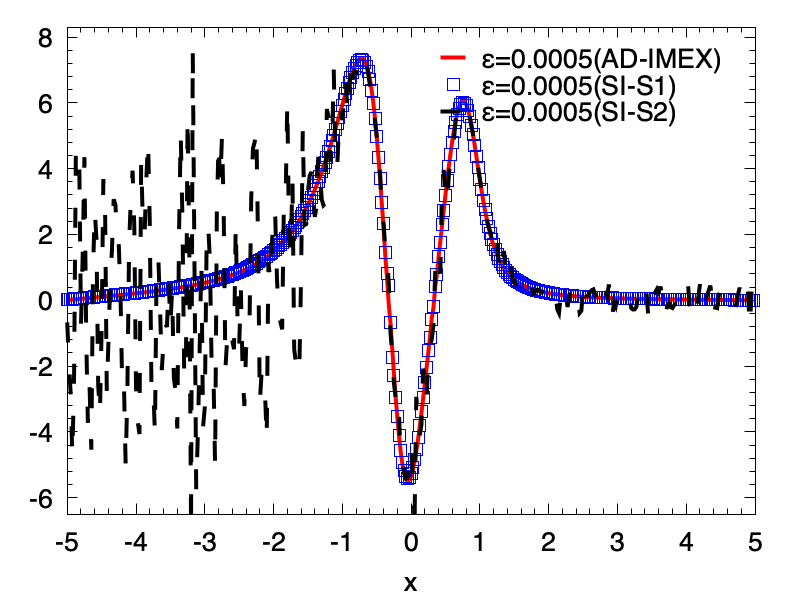}}\quad
    }
    \mbox{
        \subfigure[ $\eps=1,\, T=0.5,\,gk^2=2500$]
    {\includegraphics[width=7.0cm]{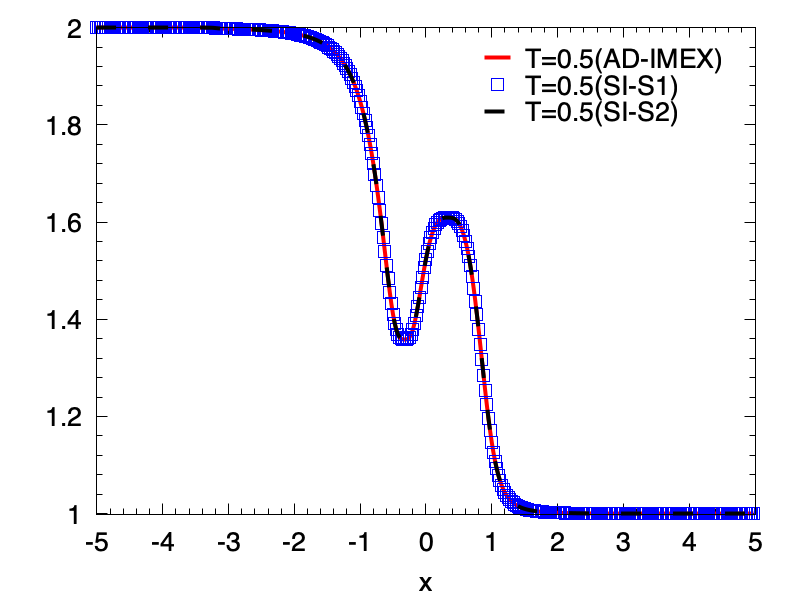}}\quad
        \subfigure[ $\eps=1,\, T=0.5,\,gk^2=2500$]
    {\includegraphics[width=7.0cm]{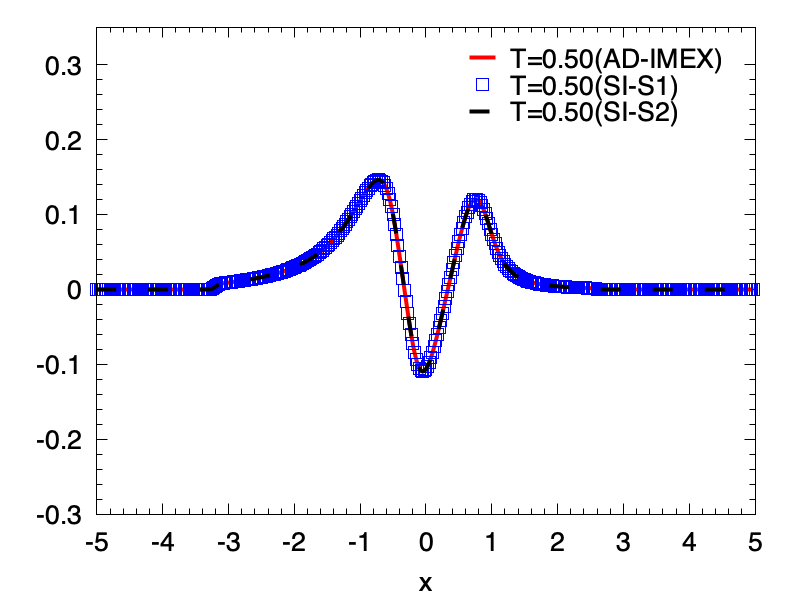}}\quad
    }
    \caption{ Example~\ref{exam55}, the numerical results about the water depth  $h$ (left) and momentum $m$ (right) for the smooth initial conditions \eqref{exam55_F1} with $N=400$ uniform meshes.}
    \label{Fig_ex55_2}
    \end{center}
    \end{figure}

    \begin{figure}[hbtp]
      \begin{center}
      \mbox{
          \subfigure[water depth $h$]
      {\includegraphics[width=7.0cm]{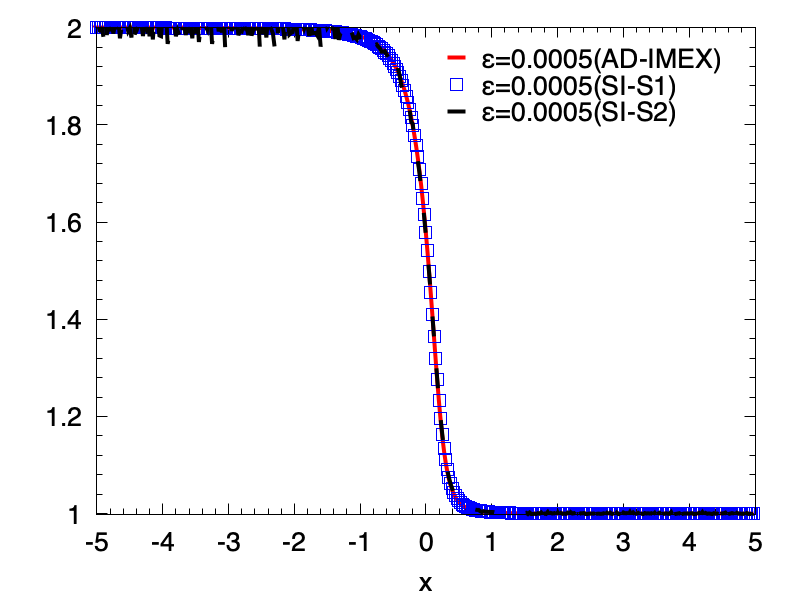}}\quad
          \subfigure[water depth $h$]
      {\includegraphics[width=7.0cm]{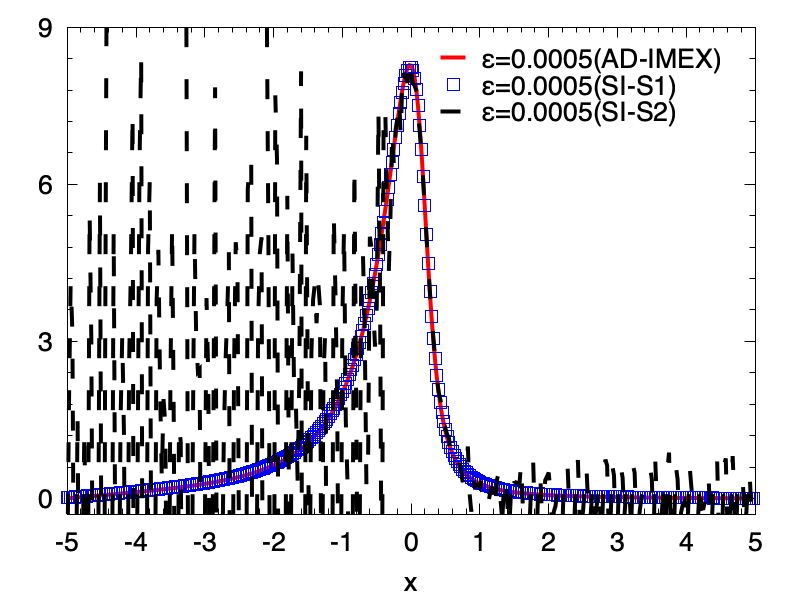}}\quad
      }
      \mbox{
          \subfigure[water depth $h$]
      {\includegraphics[width=7.0cm]{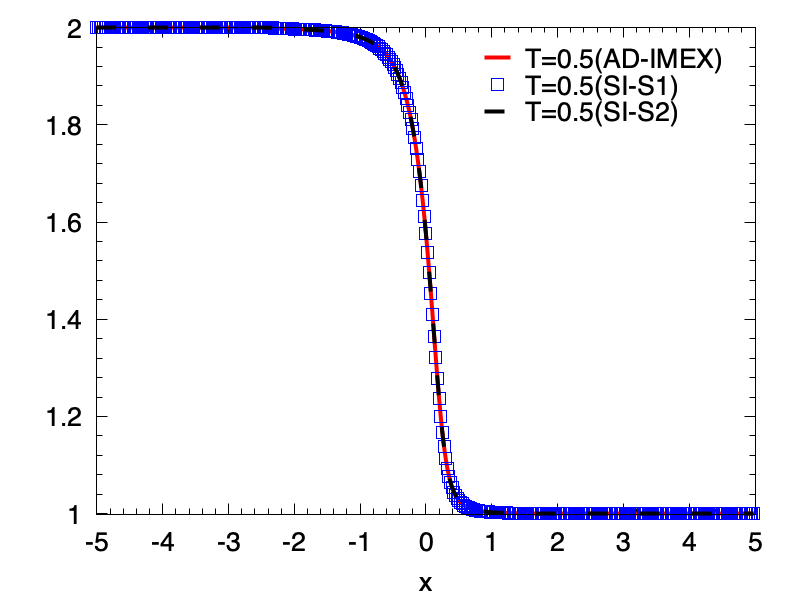}}\quad
          \subfigure[water depth $h$]
      {\includegraphics[width=7.0cm]{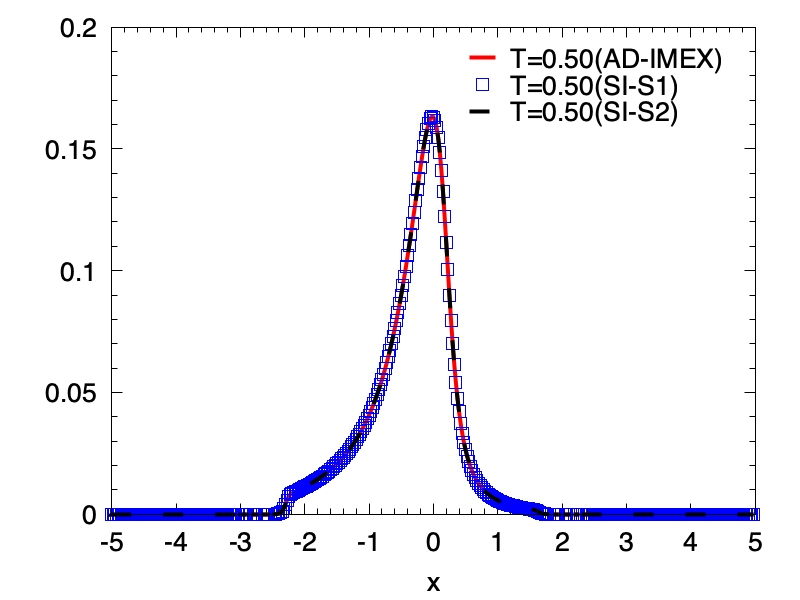}}\quad
      }
      \caption{ Example~\ref{exam55}, the numerical results about the water depth  $h$ (left)  and the momentum $m$ (right) for the discontinuous initial conditions \eqref{exam55_F2} with $N=400$ uniform meshes.}
      \label{Fig_ex55_3}
      \end{center}
      \end{figure}

      \begin{table}[htbp]
        \caption{ Example \ref{exam55}. The iteration numbers for two types of initial conditions (Unit: iterations).}
        \begin{center}
          \begin{tabular}{c|c|c|c|c|c|c}
            \hline\hline	
          \multicolumn{1}{c|}{\multirow{2}*{\text{Initial condition }}}&\multicolumn{2}{|c|}{\multirow{2}*{Case}}&\multicolumn{2}{c|}{ $N=200$}&\multicolumn{2}{c}{$N=400$}\\
            \cline{4-7}
      \multicolumn{1}{c|}{} &\multicolumn{1}{|c}{} & &AD-IMEX& SI-S1& AD-IMEX&SI-S1\\  \hline\hline
      \multicolumn{1}{c|}{\multirow{8}*{\text{{Continue ~\eqref{exam55_F1}}}}} &\multicolumn{1}{c|}{\multirow{4}*{\text{{$\theta=1$}}}}
          &$\eps=1.00$  &     41&       40&     72&       72\\ \cline{3-7}
        & &$\eps=0.20$  &     8902&     80&     14592&    125\\ \cline{3-7}
        & &$\eps=0.10$  &     11836&    130&    20948&    204\\ \cline{3-7}
        & &$\eps=0.0005$&     1230&     2499&   1995&     1216\\ \cline{2-7}
        & \multicolumn{1}{c|}{\multirow{3}*{\text{{$\eps=1$}}}}
          &$\theta=5$    &     185&      197&    368&      368\\ \cline{3-7}
        & &$\theta=10$   &     369&      409&    728&      728\\ \cline{3-7}
        & &$\theta=50$  &     1777&     2527&   3552&     3578\\ \cline{3-7}
         \hline\hline
        \multicolumn{1}{c|}{\multirow{8}*{\text{{Discontinuous~\eqref{exam55_F2}}}}} &\multicolumn{1}{c|}{\multirow{4}*{\text{{$\theta=1$}}}}
          &$\eps=1.00$  &     47&       40&     81&       80\\ \cline{3-7}
        & &$\eps=0.20$  &     8772&     113&    19616&    229\\ \cline{3-7}
        & &$\eps=0.10$  &     8121&     175&    16557&    327\\ \cline{3-7}
        & &$\eps=0.0005$&     3451&     891&    3387&     729\\ \cline{2-7}
        &\multicolumn{1}{c|}{\multirow{4}*{\text{{$\eps=1$}}}}
          &$\theta=5$   &     206&      288&    393&      458\\ \cline{3-7}
        & &$\theta=10$  &     383&      477&    753&      851\\ \cline{3-7}
        & &$\theta=50$&    1780&     2222&   3552&      3751\\ \cline{3-7}
         \hline\hline
      \end{tabular}
      \end{center}
      \label{T_ex55_1}
      \end{table}

}\end{exa}

\begin{exa}{
  \em
  \label{2D_linear_test}({\bf{2D accuracy test with linear friction}})
  We  now proceed  to test the order of accuracy for these two  schemes using  a 2D smooth problem. As mentioned  earlier, achieving the asymptotically accurate (AA) property requires well-prepared  initial conditions. So we will linearize the system by setting $\gamma=1$, which transforms  system~\eqref{SWe_MB4} into the following form:
  \begin{equation}
    \left\{
     \begin{array}{ll}
      \partial_t h + \partial_x m_1 + \partial_y m_2   = 0, \\ [3mm]
      \partial_t m_1 + \partial_x\left(\frac{m_1^2}{h}\right) + \partial_y\left(\frac{m_1m_2}{h}\right)+ \frac{1}{\eps^2}g\partial_x\left(\frac{1}{2}h^2\right)
      = -\frac{1}{\eps^2}gh\partial_x B - \frac{1}{\eps^2}m_1,\\[3mm]
      \partial_t m_2 + \partial_x(\frac{m_1m_2}{h}) + \partial_y\left(\frac{m_2^2}{h} \right) + \frac{1}{\eps^2}g\partial_y\left(\frac{1}{2}h^2\right)
      = -\frac{1}{\eps^2}gh\partial_y B - \frac{1}{\eps^2}m_2.
    \end{array}\right.
  \end{equation}
  We choose the following well-prepared initial condittions:
  \begin{equation}
    \left\{
    \begin{aligned}
      &   h(x,y,0) = \sin(\pi (x+y)) + 2, \\
      & m_1(x,y,0) = -2\pi\cos(\pi (x+y))\left(\sin(\pi (x+y)) + 2 \right),\\
      & m_2(x,y,0) = -2\pi\cos(\pi (x+y))\left(\sin(\pi (x+y)) + 2 \right).
    \end{aligned}
    \right.
\end{equation}
A flat bottom topology,  $B(x,y) = 0$, is  considered within  a  square domain $\Omega=[0,2]^2$,  with  the gravitational constant is  set to $g=2$.
The  final time is set  at $T=0.01$, and  we employ periodic boundary conditions, then divide the computational domain   into $N_k^2$ uniform meshes,  where $N_k = N_0\cdot 2^k$ with $N_0=8$ and  $k=0,1,2,3,4,5$. 
Without loss of generality, we select  three different values of $\eps$ for testing, i.e., $\eps=1,\,10^{-2},\,10^{-6}$, comparing the errors between numerical solutions and the reference solutions  obtained using   $512^2$ uniform meshes.
For the linear case, the SI-S2 scheme is equivalent to the  SI-S1 scheme, thus,   we  present only   the results obtained by the SI-S1 scheme in Table~\ref{T_2D_linear_test_1}, omitting  those from the SI-S2 scheme.

From the  table, it is observed that for  both $\eps=1$ and $\eps=10^{-6}$, the SI-S1 scheme   achieve overall third order accuracy, indicating its effectiveness  in both hyperbolic and diffusive limit.
However, for the intermediate state, i.e., $\eps=10^{-2}$,  order reductions are observed, attributed  by $\eps\sim\Delta t$ on  relatively coarse meshes.
This phenomenon  has also been reported  in \cite{boscarino2019high,boscarino2022high,huang2022high,huang2023high,huang2024FEhigh}.
\begin{table}[htbp]
    \caption{ Example \ref{2D_linear_test}. The $L_{1}$ errors and orders  for $h$, $m_1$ and $m_2$ with SI-S1 scheme.}
    \begin{center}
      \begin{tabular}{c|c|c|c|c|c|c|c}
        \hline\hline	
      \multicolumn{1}{c|}{\multirow{2}*{}}&\multicolumn{1}{|c|}{\multirow{2}*{N}}&\multicolumn{2}{c|}{ $1$}&\multicolumn{2}{c|}{$10^{-2}$}&\multicolumn{2}{c}{$10^{-6}$}\\
        \cline{3-8}
      \multicolumn{1}{c|}{} &\multicolumn{1}{|c|}{} &error& order& error&order& error&order\\  \hline\hline
      \multicolumn{1}{c|}{\multirow{5}*{$h$}}
  &  8 &     1.67E-02 &      -- &        1.18E-02 &       --&     1.19E-02 &       --\\ \cline{2-8}
  & 16 &     5.22E-03 &     1.68&        2.82E-04 &     5.39&     2.76E-04 &     5.43\\ \cline{2-8}
  & 32 &     4.89E-04 &     3.42&        1.14E-05 &     4.63&     1.15E-05 &     4.58\\ \cline{2-8}
  & 64 &     2.15E-05 &     4.51&        1.12E-06 &     3.34&     1.17E-06 &     3.30\\ \cline{2-8}
  &128 &     1.40E-06 &     3.94&        1.60E-07 &     2.81&     1.63E-07 &     2.84\\ \cline{2-8}
  &256 &     1.05E-07 &     3.74&        1.96E-08 &     3.03&     1.87E-08 &     3.12\\ \hline\hline
  \multicolumn{1}{c|}{\multirow{5}*{$hu$}}
  &  8 &     3.46E-01 &      -- &     5.50E-02 &       --&     5.50E-02 &     --\\ \cline{2-8}
  & 16 &     3.34E-02 &     3.37&     5.17E-03 &     3.41&     4.65E-03 &     3.56\\ \cline{2-8}
  & 32 &     3.13E-03 &     3.42&     2.45E-04 &     4.40&     2.19E-04 &     4.41\\ \cline{2-8}
  & 64 &     1.38E-04 &     4.50&     3.76E-05 &     2.70&     2.19E-05 &     3.32\\ \cline{2-8}
  &128 &     1.11E-05 &     3.64&     7.90E-06 &     2.25&     3.99E-06 &     2.46\\ \cline{2-8}
  &256 &     9.49E-07 &     3.54&     1.63E-06 &     2.28&     3.39E-07 &     3.56\\ \hline\hline
  \multicolumn{1}{c|}{\multirow{5}*{$hv$}}
  &  8 &     2.90E-01 &      -- &      5.50E-02 &       --&     5.50E-02 &     --\\ \cline{2-8}
  & 16 &     3.62E-02 &     3.00&      5.13E-03 &     3.42&     4.55E-03 &     3.56\\ \cline{2-8}
  & 32 &     3.43E-03 &     3.40&      2.43E-04 &     4.40&     2.19E-04 &     4.41\\ \cline{2-8}
  & 64 &     1.72E-04 &     4.32&      3.76E-05 &     2.69&     2.19E-05 &     3.32\\ \cline{2-8}
  &128 &     1.22E-05 &     3.81&      7.90E-06 &     2.25&     3.99E-06 &     2.46\\ \cline{2-8}
  &256 &     9.64E-07 &     3.67&      1.63E-06 &     2.28&     3.39E-07 &     3.56\\ \hline\hline
\end{tabular}
\end{center}
\label{T_2D_linear_test_1}
\end{table}

}
\end{exa}

\begin{exa}
{\em 
  \label{2D_Vortex}
  ({\bf{2D accuracy test with nonlinear friction}})
  In this example, we aim to test the accuracy of the  two schemes in  a nonlinear  two-dimensional setting. 
  The well-prepared initial conditions are crucial  for achieving the AA property,  while constructing  such   conditions is extremely   challenging.
  Therefore, we  consider  sufficiently smooth but not well-prepared initial conditions, given by:
  \begin{equation}
    h(\bx,0) = H_0 - \frac{\eps^2}{2}e^{-(r^2-1)},
  \end{equation}
  where $H_0=1$ and the gravitational constant $g=1$.
  Polar coordinates $(\alpha,r)$  are defined as 
  \begin{equation}
    \tan(\alpha)=\frac{y}{x}, \quad r^2 = x^2 + y^2,
  \end{equation}
  with  the corresponding  angular velocity $m_{\alpha} = re^{-\frac{1}{2}(r^2-1)}$.
  This  example is considered in a square domain $[-5,5]\times[-5,5]$ with  a small Manning friction  $k=0.001$ and  periodic boundary conditions .
  
  Given that   the initial conditions are not well-prepared,  we limit our  testing to  cases where $\eps$ is not too small, such as  $\eps = 0.32,\,10^{-2},\,10^{-3}$ .
  The final time is set to  $T=0.1$, and  we use   the  uniform  of $512 \times 512$ subdivision as the reference solutions, we then select   six different $N_k^2$ uniform subdivisions, where $N_k = N_0 \cdot 2^k$ with  $N_k=8$ and $k=0,\,1,\,3,\,4,\,5$.
  Calculating   the errors between the numerical solutions  obtained by SI-S1 and SI-S2 schemes  and the reference solutions, and  showing them   in Table~\ref{2D_Vortex_T1} and Table~\ref{2D_Vortex_T3}, respectively.
  These tables reveal that the SI-S1 scheme achieves the expected high-order accuracy when $\eps$ is not too small and is also not proportional to the time step $\Dt$.
  The  SI-S2 scheme attains the expected accuracy for $\eps=0.32$ and $10^{-2}$, but fails to achieve the high order accuracy when $\eps=10^{-3}$.
  This   indicates that  the SI-S1 scheme performs better than the  SI-S2 as $\eps$  approaches zero. 
  \begin{table}[htbp]
    \caption{ Example \ref{2D_Vortex}. The $L_{1}$ errors and orders  for $h$, $m_1$ and $m_2$ with SI-S1 scheme.}
    \begin{center}
      \begin{tabular}{c|c|c|c|c|c|c|c}
        \hline\hline	
      \multicolumn{1}{c|}{\multirow{2}*{}}&\multicolumn{1}{|c|}{\multirow{2}*{N}}&\multicolumn{2}{c|}{ $0.32$}&\multicolumn{2}{c|}{$10^{-2}$}&\multicolumn{2}{c}{$10^{-3}$}\\
        \cline{3-8}
      \multicolumn{1}{c|}{} &\multicolumn{1}{|c|}{} &error& order& error&order& error&order\\  \hline\hline
      \multicolumn{1}{c|}{\multirow{5}*{$h$}}
  &  8 &     9.69E-04 &      -- &        2.47E-04 &       --&     2.01E-04 &       --\\ \cline{2-8}
  & 16 &     3.05E-04 &     1.67&        2.08E-05 &     3.57&     1.37E-05 &     3.87\\ \cline{2-8}
  & 32 &     3.52E-05 &     3.11&        2.83E-06 &     2.88&     2.49E-06 &     2.46\\ \cline{2-8}
  & 64 &     1.60E-06 &     4.46&        8.50E-08 &     5.06&     5.07E-08 &     5.62\\ \cline{2-8}
  &128 &     6.72E-08 &     4.58&        3.76E-09 &     4.50&     3.36E-09 &     3.91\\ \cline{2-8}
  &256 &     4.55E-09 &     3.88&        3.45E-10 &     3.44&     3.61E-09 &      -- \\ \hline\hline
  \multicolumn{1}{c|}{\multirow{5}*{$m_1$}}
  &  8 &     6.70E-03 &      -- &     9.17E-03 &       --&     2.05E-02 &     --\\ \cline{2-8}
  & 16 &     1.17E-03 &     2.51&     1.80E-03 &     2.35&     1.40E-03 &     3.87\\ \cline{2-8}
  & 32 &     1.42E-04 &     3.05&     2.04E-04 &     3.14&     2.24E-04 &     2.65\\ \cline{2-8}
  & 64 &     9.24E-06 &     3.94&     1.02E-05 &     4.33&     7.59E-06 &     4.88\\ \cline{2-8}
  &128 &     4.98E-07 &     4.21&     6.19E-07 &     4.04&     4.51E-07 &     4.07\\ \cline{2-8}
  &256 &     4.52E-08 &     3.46&     7.32E-08 &     3.08&     1.48E-07 &     1.60\\ \hline\hline
  \multicolumn{1}{c|}{\multirow{5}*{$m_2$}}
  &  8 &     1.01E-02 &      -- &      1.18E-02 &       --&     2.21E-02 &     --\\ \cline{2-8}
  & 16 &     1.70E-03 &     2.57&      2.26E-03 &     2.38&     1.79E-03 &     3.62\\ \cline{2-8}
  & 32 &     2.32E-04 &     2.87&      2.63E-04 &     3.10&     2.73E-04 &     2.71\\ \cline{2-8}
  & 64 &     1.27E-05 &     4.19&      1.39E-05 &     4.24&     1.25E-05 &     4.46\\ \cline{2-8}
  &128 &     5.55E-07 &     4.51&      6.70E-07 &     4.38&     6.49E-07 &     4.26\\ \cline{2-8}
  &256 &     5.29E-08 &     3.39&      7.96E-08 &     3.07&     1.58E-07 &     2.04\\ \hline\hline
  \end{tabular}
  \end{center}
  \label{2D_Vortex_T1}
  \end{table}

  \begin{table}[htbp]
    \caption{ Example \ref{2D_Vortex}. The $L_{1}$ errors and orders  for $h$, $m_1$ and $m_2$ with SI-S2 scheme.}
    \begin{center}
      \begin{tabular}{c|c|c|c|c|c|c|c}
        \hline\hline	
      \multicolumn{1}{c|}{\multirow{2}*{}}&\multicolumn{1}{|c|}{\multirow{2}*{N}}&\multicolumn{2}{c|}{ $0.32$}&\multicolumn{2}{c|}{$10^{-2}$}&\multicolumn{2}{c}{$10^{-3}$}\\
        \cline{3-8}
      \multicolumn{1}{c|}{} &\multicolumn{1}{|c|}{} &error& order& error&order& error&order\\  \hline\hline
      \multicolumn{1}{c|}{\multirow{5}*{$h$}}
  &  8 &     9.69E-04 &      -- &        2.47E-04 &       --&     2.53E-04 &       --\\ \cline{2-8}
  & 16 &     3.05E-04 &     1.67&        2.08E-05 &     3.57&     8.30E-04 &      --\\ \cline{2-8}
  & 32 &     3.52E-05 &     3.11&        2.83E-06 &     2.88&     4.18E-04 &     0.99\\ \cline{2-8}
  & 64 &     1.60E-06 &     4.46&        8.50E-08 &     5.06&     1.03E-04 &     2.02\\ \cline{2-8}
  &128 &     6.72E-08 &     4.58&        3.76E-09 &     4.50&     2.32E-05 &     2.15\\ \cline{2-8}
  &256 &     4.55E-09 &     3.88&        3.45E-10 &     3.44&     2.40E-06 &     3.28\\ \hline\hline
  \multicolumn{1}{c|}{\multirow{5}*{$m_1$}}
  &  8 &     6.70E-03 &      -- &     9.17E-03 &       --&     1.77E-02 &     --\\ \cline{2-8}
  & 16 &     1.17E-03 &     2.51&     1.80E-03 &     2.35&     6.38E-02 &     --\\ \cline{2-8}
  & 32 &     1.42E-04 &     3.05&     2.04E-04 &     3.14&     2.37E-02 &     1.43\\ \cline{2-8}
  & 64 &     9.24E-06 &     3.94&     1.02E-05 &     4.33&     7.15E-03 &     1.73\\ \cline{2-8}
  &128 &     4.98E-07 &     4.21&     6.19E-07 &     4.04&     2.47E-03 &     1.54\\ \cline{2-8}
  &256 &     4.52E-08 &     3.46&     7.32E-08 &     3.08&     1.01E-03 &     1.29\\ \hline\hline
  \multicolumn{1}{c|}{\multirow{5}*{$m_2$}}
  &  8 &     1.01E-02 &      -- &      1.18E-02 &       --&     1.95E-02 &     --\\ \cline{2-8}
  & 16 &     1.70E-03 &     2.51&      2.26E-03 &     2.38&     6.29E-02 &     --\\ \cline{2-8}
  & 32 &     2.32E-04 &     2.87&      2.63E-04 &     3.10&     2.11E-02 &     1.58\\ \cline{2-8}
  & 64 &     1.27E-05 &     4.19&      1.39E-05 &     4.24&     7.87E-03 &     1.42\\ \cline{2-8}
  &128 &     5.55E-07 &     4.51&      6.70E-07 &     4.38&     2.45E-03 &     1.68\\ \cline{2-8}
  &256 &     5.29E-08 &     3.39&      7.96E-08 &     3.07&     9.99E-04 &     1.30\\ \hline\hline
  \end{tabular}
  \end{center}
  \label{2D_Vortex_T3}
  \end{table}

}\end{exa}

\begin{exa}
{\em
\label{exam2D_2}
({\bf {A small perturbation of 2D steady-state water}})
In this example, we aim to  evaluate  the performance of these two  schemes in a 2D setting for capturing the propagation of small perturbations near the equilibrium state, as studied in \cite{yang2021high,huang2023high}.
The test involves an isolated elliptical-shaped hump in the  bottom topography, defined as: 
    \begin{equation}
    \label{Ex8_1}
    B(x,y) = 0.8e^{-5(x-0.9)^2-50(y-0.5)^2},
    \end{equation}
and the initial conditions are specified as 
\begin{subequations}
    \begin{equation}
    \label{Ex8_2}
    h(x,y,0)+B(x,y) = \left\{
    \begin{aligned}
    & 1+0.01,  &\text{if}\quad 0.05  \le x \le 0.15; \\
    & 1,       &\text{otherwise}.
    \end{aligned}
    \right.
    \end{equation}
    \begin{equation}
    \label{Ex8_3}
    m_1 = m_2 = 0,
    \end{equation}
\end{subequations}
These conditions are  considered within  rectangular domain $[0,2]\times[0,1]$, with inflow and   outflow boundary conditions in the x-direction and periodic boundary conditions in the y-direction.
A small friction term is also added, with a Manning coefficient of  $k=0.09$.
The numerical results for the  surface level $H=h+B$  are shown in Fig.~\ref{Fig_ex2D_2_1} using    grid subdivisions of $200\times100$.
From the results, we can see that the  initial perturbation is separated into two waves propagating to the left and right.
As the left-propagating wave exits  the domain, the right-propagating wave interacts with the non-flat bottom topography, and is well captured by these two  schemes.
The numerical results obtained from  these two  schemes are consistent with  those  presented in \cite{yang2021high,huang2023high}, demonstrating  that both schemes possess  the well-balanced property.
\begin{figure}[hbtp]
    \begin{center}
        \mbox{ \subfigure[surface level at $t=0.12$]					
             {\includegraphics[width=7cm]{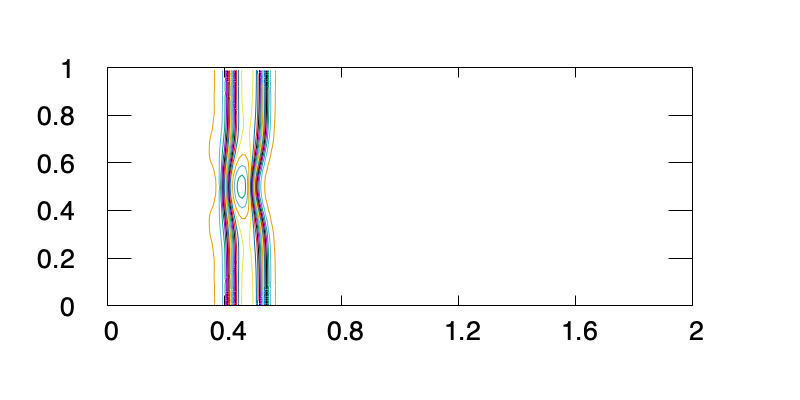}}\quad
                \subfigure[surface level at $t=0.12$]		
		     {\includegraphics[width=7cm]{eg_8_21_tend012.png}}
				}
        \mbox{ \subfigure[surface level at $t=0.24$]					
        {\includegraphics[width=7cm]{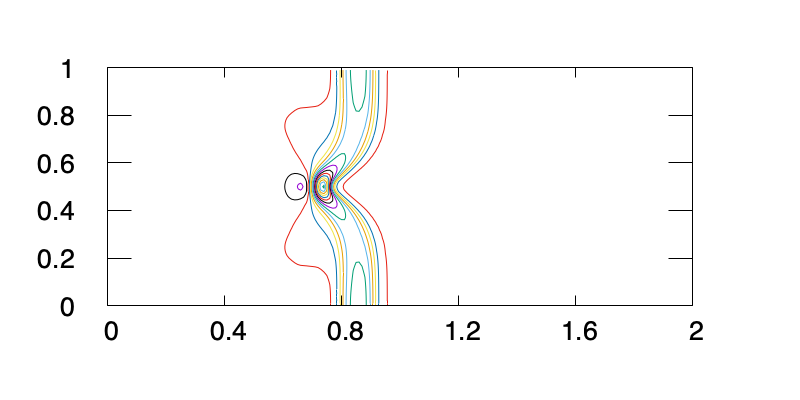}}\quad		
           \subfigure[surface level at $t=0.24$]		
    {\includegraphics[width=7cm]{eg_8_21_tend024.png}}
				}
        \mbox{ \subfigure[surface level at $t=0.36$]					
        {\includegraphics[width=7cm]{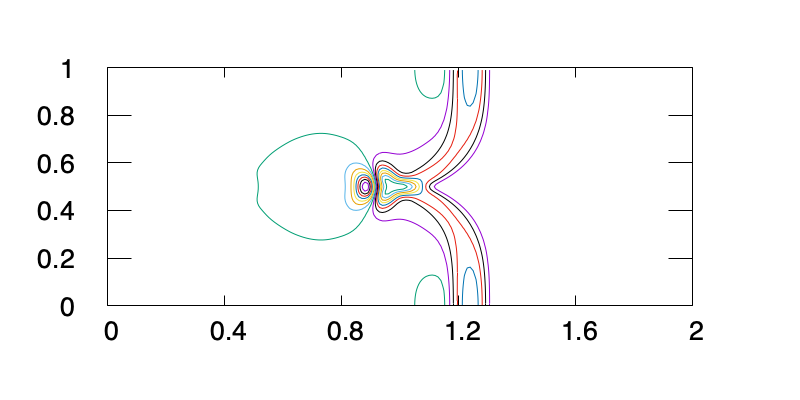}}\quad		
           \subfigure[surface level at $t=0.36$]		
    {\includegraphics[width=7cm]{eg_8_21_tend036.png}}
				}
        \mbox{\subfigure[surface level at $t=0.48$]					
        {\includegraphics[width=7cm]{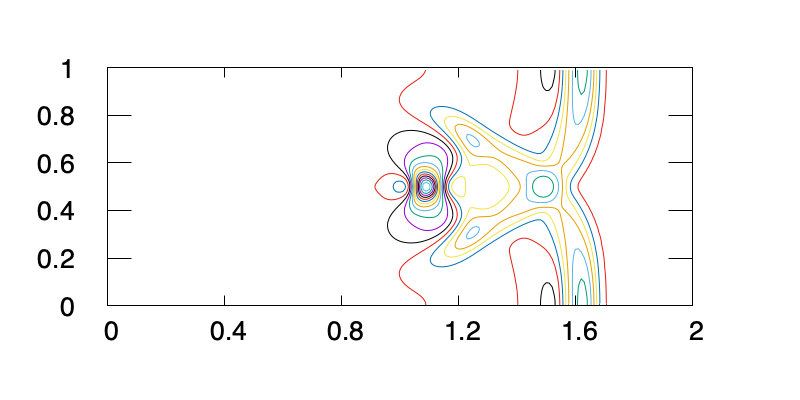}}\quad		
           \subfigure[surface level at $t=0.48$]		
    {\includegraphics[width=7cm]{eg_8_21_tend048.png}}
				}
        \mbox{\subfigure[surface level at $t=0.60$]					
        {\includegraphics[width=7cm]{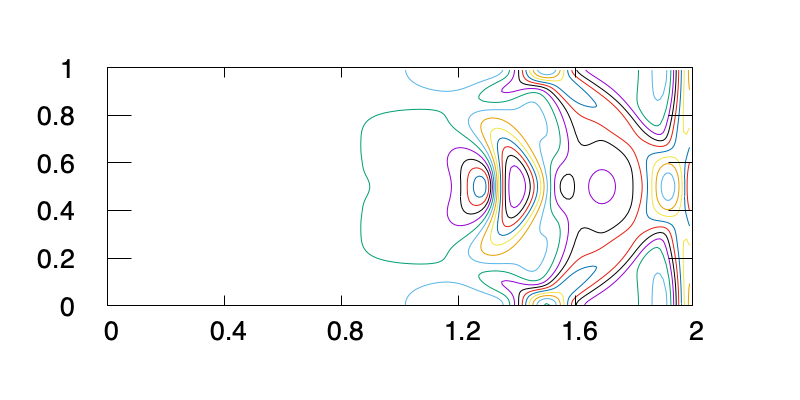}}\quad		
           \subfigure[surface level at $t=0.60$]		
    {\includegraphics[width=7cm]{eg_8_21_tend06.png}}
				}
				\caption{ Numerical solutions of the surface level $h+b$ for Example~\ref{exam2D_2} with $200\times100$ uniform meshes.
From top to bottom: at $t=0.12$ from 0.9998 to 1.0060; at  $t=0.24$ from 0.9967 to 1.0130;
at  $t=0.36$ from 0.9901 to 1.0097; $t=0.48$ from 0.9906 to 1.0043; $t=0.6$ from 0.9955 to 1.0045. 30 contour lines are used. Left: SI-S1 scheme;  Right: SI-S2 scheme. }
				\label{Fig_ex2D_2_1}
			\end{center}
		\end{figure}

}
\end{exa}

\section{Conclusion}
\label{sec6}
\setcounter{equation}{0}
\setcounter{figure}{0}
In this paper, we propose two types of  high-order semi-implicit asymptotic preserving  finite difference WENO schemes for the shallow water equations with  bottom topology and Manning friction term \eqref{SWe_MB4}. 
The  objective is to develop  a highly  efficient method that avoids the use of weighted diffusive terms, as proposed in \cite{huang2023high}.
Theoretical analysis and numerical tests demonstrated  that the  SI-S1 scheme has the AP and AA, and  well-balanced properties.
Unfortunately, the SI-S2 scheme  fails to converge to the limiting equations~\eqref{lim_E4} as $\eps$ approaches zero, indicating that the implicit treatment of the Manning friction is  necessary for  the stability of the scheme.
Furthermore, comparing to the first-order asymptotic preserving scheme proposed in \cite{bulteau2020fully}, our schemes  not only avoid  the severe time step restrictions as $\eps\to 0$ but also offer higher resolution,  as demonstrated by the numerical experiments. 
Furthermore,  the schemes developed in this work are more efficient than those proposed  in our previous work~\cite{huang2023high}, essentially in the intermediate state.

\section{Acknowledgment} The first and third authors are partially supported by National Key R \& D Program of China No.
2022YFA1004500, NSFC grant No. 92270112, NSF of Fujian Province No. 2023J02003. Sebastiano Boscarino is
supported for this work by (1) the Spoke 1 “FutureHPC \& BigData” of the Italian Research Center on High-
Performance Computing, Big Data and Quantum Computing (ICSC) funded by MUR Missione 4 Componente
2 Investimento1.4:Potenziamento strutture di ricerca e creazione di “campioni nazionali di R\&S (M4C2-19
)”; by (2) the Italian Ministry of Instruction, University and Research (MIUR) to support this research with
funds coming from PRIN Project 2022 (2022KA3JBA),entitled “Advanced numerical methods for time dependent 
parametric partial differential equations and applications”; (3) from Italian Ministerial grant PRIN2022 PNRR
“FIN4GEO: Forward and Inverse Numerical Modeling of hydrothermal systems in volcanic regions with application to geothermal energy exploitation.”,(No.P2022BNB97).S.Boscarino is a member of the INdAM Research group GNCS.

\bibliographystyle{abbrv}
\bibliography{refer}
\end{document}